\documentclass[psamsfonts]{amsart}
\usepackage{
graphicx,
fancyhdr,
amsmath,
amsfonts,
dsfont,
amsthm,
amssymb,
enumerate,
mathrsfs,
hyperref,
xcolor,
ulem,
cancel,
enumitem,
cleveref
}

\newcommand{\epf}{{{\hfill $\Box$ \smallskip}}}

\newcommand{\stepdensity}{p}
\newcommand{\E}{\mathbb{E}}
\newcommand{\R}{\mathbb{R}}

\newcommand{\Z}{\mathbb{Z}}

\newcommand{\Q}{\mathbb{Q}}
\newcommand{\N}{\mathbb{N}}

\newcommand{\1}{\mathds{1}}

\newcommand{\Prb}{\mathbb{P}}

\newcommand{\dx}[1]{\text{d}{#1}}

\newcommand{\mushear}{\tilde{\mu}}
\newcommand{\Zshear}{\tilde{Z}}

\newcommand{\stepm}{\nu}
\newcommand{\refm}{\lambda}
\newcommand{\refpm}{\mathcal{M}}

\newtheorem{proposition}{Proposition}[section]

\newtheorem{lemma}{Lemma}[section]
\newtheorem{theorem}{Theorem}[section]

\theoremstyle{definition}

\renewcommand*{\theassumption}{\Alph{assumption}}
\newtheorem{remark}{Remark}

\newlist{assumpRequirements}{enumerate}{10}
\setlist[assumpRequirements]{label*=\Roman*}
\setlist[assumpRequirements,1]{label=(\Roman*), ref = \theassumption.\Roman*}
\crefname{assumpRequirementsi}{assumption}{assumptions}

\newlist{requirements}{enumerate}{10}
\setlist[requirements]{label*=\alph*}
\setlist[requirements,1]{label=\rm(\alph*), ref = \rm(\alph*)}
\crefname{requirementsi}{requirement}{requirements}

\numberwithin{equation}{section}

\author{Yuri Bakhtin and Douglas Dow}
\title[Differentiability of the Shape Function in Continuous Space]{Differentiability of the Shape Function for Directed Polymers in Continuous Space}

\begin{document}

\begin{abstract}
For directed polymers, the shape function computes the limiting average energy accrued by paths with a given average slope.
We prove that, for a large family of directed polymer models in discrete time and continuous space in dimension $1+1$,
for positive and zero temperature, the shape function is differentiable with respect to the slope on the entire real line. 
\end{abstract}

\address{Courant Institute of Mathematical Sciences, New York University}

\maketitle

\section{Introduction}

\subsection{The main result.}
The main goal of this paper is to prove the following result (see \Cref{thm:differentiabilityZeroTemp,thm:differentiabilityPolymer} for the precise statements):

\begin{theorem}
\label{th:main_high_level}
For a large family of directed polymer models in $1+1$ dimension
in continuous space and discrete time,
for positive and zero temperature, the shape function is differentiable on the entire real line.
\end{theorem}

The term {\it directed polymers} refers to a class of probabilistic models for chains of monomers interacting with a dynamic random environment and with their nearest 
neighbors in the chain. At zero temperature, the chains are arranged to minimize the total energy of those interactions, and at positive temperature the chains are drawn from a Gibbs distribution corresponding to the same energy. Most of the existing literature concerns lattice models, see, e.g., \cite{Comets:MR3444835}, \cite{Janjigian--Rassoul-Agha:MR4089495}, \cite{Bates-Chatterje:MR4089496},
\cite{JRAS-JEMS},
 and references therein. Continuous space models have also been studied and lately have been gaining popularity in connection with stochastic PDEs such as the KPZ equation and its relatives. Here is a sample of literature on continuous models in positive and zero temperature:
 \cite{Aldous-Diaconis:MR1355056}
\cite{Comets-Yoshida:MR2117626}, \cite{CaPi}, \cite{BCK:MR3110798}, \cite{Alberts-Khanin-Quastel:MR3162542},  
\cite{kickb:bakhtin2016},  \cite{BK18},  \cite{Bakhtin-Li:MR3911894}, \cite{bakhtin-seo2020localization}, \cite{Bakhtin-Dow:joint-loc},
\cite{JRAS:https://doi.org/10.48550/arxiv.2211.06779},
\cite{BSS-stationary-horizon}, \cite{Das-Zhu:https://doi.org/10.48550/arxiv.2203.03607}.

Differentiability and strong convexity of the shape function for generic
polymer models without exact solvability or symmetries
have been long-standing open problems. Our main theorem
gives a solution of the differentiability problem in the space-continuous case. Our proof exploits invariance in distribution of the underlying environment under shear transformations, a consequence of working in continuous space and of standard stationarity assumptions on the environment. We impose no shear invariance assumptions on the 
self-interaction energy, which means that we work in a very broad setup. Still, there is a number of similar continuous Last Passage Percolation (LPP) and First Passage Percolation (FPP) models that do not belong to our setting formally.  Our argument can easily be adjusted to treat these models as well, as we discuss below.

\subsection{Polymers and shape functions.}
\label{sec:intro-poly}
Let us introduce a fairly general setup based on the Gibbs formalism for random walks in random potentials
embracing several polymer models in discrete time in dimension  $1+1$.

The  i.i.d.\ steps of the random walk all have the same distribution given by a Borel probability 
measure $\stepm$ on~$\R$. This measure is usually assumed to be of the form $\stepm(dx)=p(x)\refm(dx)$, i.e.,
it has density $p\ge 0$ with respect to  the measure~$\refm$ which is either
the counting measure on $\Z$ or Lebesgue measure on $\R$. These choices of $\refm$ give rise, respectively, to lattice polymer models or
continuous space polymers. 
In this paper, we work only with the space continuous case but here we include the lattice case for the historical context.

In both cases, let us define  
\begin{equation}
\label{eq:V-from-p}
V(x)=\begin{cases}
-\log p(x),& p(x)>0,\\
+\infty, & p(x)=0,
\end{cases}
\end{equation}
which can be viewed as the energy of self-interaction between neighbors in the polymer chain. 

The external environment is given by a potential with random realization $F$ that is a function from $\Z\times \R$ to $\R$. 
A path passing through a point $x$ at time $k$ picks up energy $F_k(x)$ from that point.
The most interesting situation is where the random field $F$ is space-time stationary. It is natural to require that $(F_k)_{k\in\Z}$ is an i.i.d.\ sequence.

For a fixed realization of the potential $F$, the total energy or action of a path $\gamma:\{m,m+1,\ldots,n\}\to\R$ given by
\begin{equation}
    \label{eq:path_energy}
    A^{m,n}(\gamma) = \sum_{k=m}^{n-1}V(\Delta_k\gamma) + \sum_{k=m}^{n-1}F_k(\gamma_k)
\end{equation}
combines
the energies of self-interaction and interaction with the environment.  
Here
\[
\Delta_k\gamma=\gamma_{k+1}-\gamma_k.
\]
If~the denisty $p$ is Gaussian, i.e., $V$ is quadratic, then the first term in~\eqref{eq:path_energy} can be viewed as the kinetic action or the total
kinetic energy of a particle traveling along the path~$\gamma$. For general~$V$,  the total energy $A^{m,n}(\gamma)$ is the sum
of the generalized kinetic and potential energies. 

The point-to-point (p2p) directed polymer measure $\mu^{m,n}_{x,y}$ on paths
$\gamma$ connecting the point $x\in\R$ at time $m\in\Z$ and a point $y\in\R$  at time $n\in\Z$ is defined as the Gibbs distribution with reference measure 
\begin{equation}
\label{eq:ref-measure}
\refpm^{m,n}_{x,y}(d\gamma)=\delta_x(d\gamma_m)\refm(d\gamma_{m+1})\ldots \refm(d\gamma_{n-1})\delta_{y}(d\gamma_n)
\end{equation} 
($\delta_x$ means the Dirac mass concentrated at $x$)
and energy $A^{m,n}$, which is random since it  depends on the realization of $F$. More precisely,
\begin{align}
\label{eq:p2p_meas}
\mu^{m,n}_{x,y}(d\gamma)
=&\frac{1}{Z^{m,n}_{x,y}}e^{-A^{m,n}(\gamma)}  \refpm^{m,n}_{x,y}(d\gamma)
\\ \notag
=&\frac{1}{Z^{m,n}_{x,y}}e^{-\sum_{k=m}^{n-1}V(\Delta_k \gamma) - \sum_{k=m}^{n-1}F_k(\gamma_k)}
\refpm^{m,n}_{x,y}(d\gamma),
\end{align}
where the normalizing constant $Z^{m,n}_{x,y}$, called the {\it partition function}, is given by
\begin{align}
\label{eq:p2p_pf}
Z^{m,n}_{x,y}=&\int_{\R\times\R^{n-m-1}\times\R} e^{-A^{m,n}(\gamma)}  \refpm^{m,n}_{x,y}(d\gamma)
\\
\notag
=&
\int_{\R\times\R^{n-m-1}\times\R}e^{-\sum_{k=m}^{n-1}V(\Delta_k \gamma) - \sum_{k=m}^{n-1}F_k(\gamma_k)}\refpm^{m,n}_{x,y}(d\gamma).
\end{align}

Of course,  the definition \eqref{eq:p2p_meas} is valid only if
$0<Z^{m,n}_{x,y}<\infty$.

All properties of polymer measures can be expressed in terms of this family of partition functions. In particular,
a basic object in the study of polymers of growing and infinite length is the {\it free energy density} or {\it shape function} defined by
\begin{equation}
\label{eq:shape_positive_temp}
\Lambda_1(v)  = -\lim_{n\to\infty} \frac{1}{n}\log Z^{0,n}_{0,vn}.
\end{equation}
A subadditivity argument implies that under fairly broad conditions on $V$ and $F$ in the space-continuous setting (in particular $V$ is required to be finite everywhere), 
there is a deterministic convex $\R$-valued function~$\Lambda_1$ such that  
for all $v\in\R$,
\eqref{eq:shape_positive_temp}
holds
almost surely. In the lattice case, or if~$V$ takes infinite values, 
a version of this still holds true, although one needs to replace the sequence of points $vn$ on the right-hand side of~\eqref{eq:shape_positive_temp}  by a
sequence of points $x_n$ in the support of the endpoint distribution of the random walk satisfying $x_n/n\to v$, and one may need to allow for infinite values of~$\Lambda_1(v)$.

The definitions~\eqref{eq:p2p_meas} and \eqref{eq:p2p_pf} implicitly use the temperature value $T=1$. Any other positive value of temperature can be used and treated in the same way. In the limit $T\downarrow 0$, the finite volume Gibbs measures
concentrate near paths realizing the minimum value
\begin{equation}\notag
    A^{m,n}_{x,y}=\inf_{\substack{\gamma_m=x\\ \gamma_n=y}} A^{m,n}(\gamma).
\end{equation}
These minimal actions and the minimizers, also known as geodesics, have also been studied in the literature under the title
of last passage percolation (LPP). Similarly to the positive temperature case,  the {\it shape function} can be defined for $T=0$ by
\begin{equation}
\label{eq:shape_for_0_temp}
\Lambda_0(v)  = \lim_{n\to\infty} \frac{1}{n} A^{0,n}_{0,vn}.
\end{equation}
In fact, similarly to the positive temperature case, a subadditivity
argument implies that 
under fairly general conditions on $V$ and $F$, there is a deterministic convex function $\Lambda_0$ such 
that for all $v\in\R$, \eqref{eq:shape_for_0_temp} holds almost surely, with necessary modifications for lattice models and energy functions $V$ that are allowed to take infinite values.

In both cases, shape functions govern the scaling limit of the model
of the Law of Large Numbers type.
This can be viewed as a homogenization statement. In this viewpoint, the asymptotic behavior of the model at large scales is described by a Hamilton--Jacobi--Bellman  (HJB) equation with effective Hamiltonian that is the Legendre dual to the effective Lagrangian given by the shape function, see
\cite{Kosygina-Varadhan:MR2400607}, \cite{BK18}, \cite{JRAS:https://doi.org/10.48550/arxiv.2211.06779}. 
In fact, the action  $A^{m,n}$ can be used to define a discrete time random Lax--Oleinik semigroup which can be viewed as the solution operator for a HJB
equation with random forcing in discrete time.
We note that the most studied setup for stochastic homogenization of HJB equations involves a static environment, not a constantly refreshing dynamic environment as in our case, where the potential~$F$ is i.i.d., or {\it white}, in time.
 
Besides convexity, little is known about the shape functions $\Lambda_0$ and $\Lambda_1$.  There are only a few situations where these functions are known precisely or up to a constant. These are situations with a useful symmetry group or even precise algebraic properties.  
In \cite{kickb:bakhtin2016}, \cite{Bakhtin-Li:MR3911894},  for the quadratic energy $V(x)=a+bx^2$ corresponding to Gaussian densities $p$, both positive and zero temperature shape functions are shown to also be quadratic if $(F_k)_{k\in\Z}$  form an i.i.d.\ sequence. This is a consequence of the shear invariance of the model under these assumptions. Namely, under the shear transformations applied to paths and the environment, the distribution of the environment is invariant, polymer measures are mapped into polymer measures, minimizers are mapped into minimizers and the partition functions and optimal actions undergo a simple algebraic transformation.  

Similar results have been obtained for continuous time versions of this model  with quadratic action and white-in-time $F$
allowing for shear invariance:
 in the zero temperature model for inviscid Burgers equation with Poissonian forcing
\cite{BCK:MR3110798}, where the shear invariance was exploited for  the first time in this context, and a positive temperature version in 
\cite{JRAS:https://doi.org/10.48550/arxiv.2211.06779} for the KPZ equation. Besides the shear invariance, the proof in~\cite{JRAS:https://doi.org/10.48550/arxiv.2211.06779} uses the Gaussian specifics of the model allowing to compute the exact constant term in the quadratic shape function.
More shear-invariant models can easily be introduced. For example, Brownian polymers in Poissonian environment of \cite{Comets-Yoshida:MR2117626} belong to this class, and one can consider various versions of this model or the model of \cite{BCK:MR3110798} with 
compound Poisson/L\'evy environment in continuous or discrete time.

In the exponential  LPP lattice model, also known as the exponential corner growth model, the random walk is a simple symmetric walk on integers and the potential is given by i.i.d.\ exponential random variables. 
 An explicit shape function for this zero temperature model is an immediate implication of the main result of~\cite{Rost:MR635270} on the TASEP Markov process. A similar picture holds for i.i.d.\ geometric weights. 

The (weighted) Hammersley process of optimal up-right paths through clouds of Poisson points also belong to the LPP family. The setting we describe here needs to be adjusted to include those models.  
An explicit form of the shape function is known for them (up to a constant  unknown for general weights) which is also
a result of existence of a group of transformations preserving the optimal action and the distribution of Poissonian environment, see
\cite{Hammersley:MR0405665}, \cite{Aldous-Diaconis:MR1355056}, \cite{CaPi}.

The shape function is explicitly computed for the exactly solvable log-gamma polymer in \cite{Seppalainen:MR2917766}.

The O'Connell--Yor directed polymers in Brownian environment also allow for an explicit shape function for zero temperature
\cite{Baryshnikov:MR1818248},\cite{Gravner-Tracy-Widom:MR1830441},\cite{Hambly-Martin-O'Connell:MR1935124}
and positive temperature \cite{Moriarty-O'Connell:MR2343849}. The analysis is based on exact computations exploiting connections with random matrix theory and queueing theory.

For Euclidean First Passage Percolation (FPP) models, the rotational symmetry of the shape function is inherited from
the rotational symmetry of the model, see~\cite{HoNe3}.

In all these cases with explicit formulas for shape functions, they are differentiable on their domain and strongly convex, 
i.e., they allow  for a uniform lower bound on the curvature. It is broadly believed that these properties hold for a very large class of models essentially coinciding with the KPZ universality class. However, the strong convexity conjecture remains open. 

\subsection{Differentiability}
Since $\Lambda_0$ and $\Lambda_1$ are convex, their differentiability may fail at most at countably many $v\in \R$.

In the present paper, in the space-continuous polymer setting, for both zero and positive values of temperature, we introduce a set of mild conditions on the energy~$V$ and space-time stationary random potential $F$ that allow us to prove our main result stating differentiability everywhere for $\Lambda_0$ and $\Lambda_1$, 
Theorem~\ref{th:main_high_level}
(see Section~\ref{sec:Setting_MainResults} for precise statements of the conditions and results). Our result also gives a formula for the derivative in terms of limiting statistics of the model.

Differentiability of a convex function automatically implies that the derivative is a continuous nondecreasing function, so, in fact, this means that $\Lambda_0,\Lambda_1\in C^1(\R)$.

No such result has been available for any other model except the short finite list above where the derivative of the shape function can be computed explicitly. In contrast, our theorem makes a universality claim and asserts differentiability for a
large class of models.

The only existing differentiability result that holds for a generic class of models that is known to us is still far less powerful than ours. The main theorem from \cite{AuffingerDamron_PercolationCone} describes a class of random i.i.d.\  potentials on $\Z^2$ such that the limit shapes for the associated FPP and LPP models have a flat edge and are differentiable at the endpoints of that edge.

Differentiability of the shape function has important consequences for large scale behavior of polymers and geodesics.
In \cite{BCK:MR3110798},\cite{kickb:bakhtin2016},\cite{Bakhtin-Li:MR3911894}, and related papers, it was strong convexity 
of the (explicitly known) shape function that  was used to prove
the almost sure existence and uniqueness of 
one-sided geodesics (action minimizers) and
one-sided Dobrushin--Lanford--Ruelle (DLR) measures, or infinite-volume polymer measures (IVPM's) with an arbitrary asymptotic slope (also see \cite{Bakhtin-HBChen:MR4421173} for the description of those IVPM's as invariant measures of a stochastic heat equation in random environment).
But for lattice polymers,
 \cite{Janjigian--Rassoul-Agha:MR4089495} and \cite{JRAS-JEMS} derive existence and uniqueness of a one-sided geodesic or 
 an IVPM either with a prescribed asymptotic slope $v$ (if $v$ does not belong to any flat edge of the shape function) or with partial asymptotic slopes within a given flat edge, from the differentiability of shape function
 (which is still only an assumption and a plausible conjecture for lattice models).

In \cite{DamonHanson2017}, it was shown that absence of bigeodesics
in FPP is a consequence of 
the shape function differentiability assumption.

In general, the locally quadratic behavior of the shape function combining differentiability and strong convexity is believed to put the
model into the KPZ universality class, although at this point only heuristic arguments are available for this claim, see~\cite{BK18}. Though we do not establish strong or strict convexity of the shape function, our formula for the derivative of shape function provides heuristic evidence in support of the conjecture of strict convexity of the shape function for strictly convex~$V$.

In this paper, we focus on differentiability. We will work on its consequences in our space-continuous setting in forthcoming publications.

\subsection{Our assumptions, generality of our arguments, and other models}

The setting we work in is quite broad but in order to simplify the exposition
and stress the main ideas we did not try to obtain the most general assumptions. We stress though that the space-continuous setting is important to us and we claim no new result for the lattice case.

 We assume that the random potential is white (i.i.d.)
in time and stationary in space. We do not require ergodicity with respect to spatial shifts. For the one-dimensional marginals $F_k(x)$, $(k,x)\in\Z\times\R$, we require
a tail condition at $+\infty$ and boundedness from below. The latter
assumption can be relaxed. It is introduced just to make the finiteness of the limit in the subadditive ergodic theorem obvious.

As for the nearest neighbor self-interaction energy $V$, we require it 
to be finite, $C^2$-smooth and impose some conditions on its growth at infinity. In particular, our assumptions hold for any polynomial of even order $2q$,  $q\in \N,$ 
\begin{equation}\label{eq:exampleKineticEnergies}
    V(x) = \sum_{k=0}^{2q} a_k x^k
\end{equation}
with $a_0,\dots, a_{2q}\in \R$ and positive coefficient $a_{2q}$. 

Notably, besides the $C^2$ requirement, our conditions on $V$ concern the behavior at infinity, and hence if $V$ satisfies our conditions so will $V+\eta$ for any $\eta\in C^2(\R)$ satisfying suitable decay conditions at infinity.

Under our assumptions, shear-invariance holds for 
the distribution of the environment but not for the action, its minimizers, and polymers.

Our proof exploits this shear-invariance. However, we must
tackle the fact that under shear transformations, $V$ does not transform as neatly as in the quadratic case. 

Our arguments apply to much broader situations where the environment is shear invariant, which essentially means it is white and stationary in time and space. In particular, they apply to a large family of HJB equations with Poissonian/L\'evy or Gaussian space-time white noise in discrete or continuous time. In the latter case, the energy of a path is given by an integral version of the summation in 
\eqref{eq:path_energy}. Our results also hold for a class of models based on 
point processes with long-range spatial dependence such as random lattices as long as these processes are white in time and satisfy a technical assumption
guaranteeing the linear growth of action is satisfied.

We believe that a version of our proof should work in higher dimensions and for generalized HJB-polymers introduced in~\cite{BK18}.

Our arguments also apply to anisotropic Euclidean FPP models generalizing those introduced in~\cite{HoNe3} or their version recently 
introduced in~\cite{Bakhtin-Wu:MR3949963}. The environment in these models is given by the Delaunay
triangulation with vertices at
Poissonian points sampled with constant Lebesgue intensity on the plane. For a path $x_0,x_1,\ldots, x_n$ on this graph,
the action/energy is given by
\[
E(x_0,x_1,\ldots, x_n)=\sum_{k=0}^{n-1}V(x_{k+1}-x_k),
\]  
for a smooth nearest neighbor interaction energy $V:\R^2\to\R$.
 For radially symmetric functions $V(x)=V(|x|)$, the shape function is also radially symmetric, i.e., the boundary of the limit shape is a Euclidean circle. 
For generic $V$,  smoothness of the limit shape can be obtained by our method with the help of the rotational invariance of the environment (replacing the
shear invariance employed in the present paper) and tracking how $V$ transforms under rotations.

Let us note that the white-in-time requirement on the potential is
essential in our setup not only because it leads to the shear-invariance of the
environment but also because 
general ergodic non-i.i.d.\ environments may lead to polygonal limit
shapes with corners. This has been known since~\cite{Haggstrom-Meester:MR1379157}, with some modern additions in the context of homogenization for PDEs being~\cite{Ziliotto:MR3684310} 
and~\cite{Bakhtin-Li:MR4547558}. Nevertheless, our argument is also applicable
to some non-i.i.d.\ cases including HJB equations in continuous time where the sheared environment can be
efficiently coupled to a distributional copy of the original environment.

In the next section, we introduce the setting in detail, give a rigorous statement of Theorem~\ref{th:main_high_level}, and explain how the rest of the paper is structured. 

\textbf{Acknowledgements.}  YB thanks NSF for partial support via grant DMS-1811444.  We are thankful to Aria Halavati and Keefer Rowan for stimulating discussion with regard to \Cref{secondDerivReqImprovement}.

\section{The setting and statements of main results}\label{sec:Setting_MainResults}

\subsection{Zero temperature setting and main results}
\label{sec:zeroTempModelSetUp}
 We work in the space-continuous setting throughout the paper. In particular, we assume that $\refm(dx)=dx$, i.e.,  
it is the Lebesgue measure on $\R$.

We will need the random potential $F$ to be a space-time stationary field bounded from below. 
So we will work on the probability space $(\Omega, \mathcal{F},\Prb)$, where $\Omega$ is the space of continuous functions $F:\Z\times \R\to [M_F,+\infty)$ 
endowed with local uniform topology and~$\mathcal{F}$ is the completion of the Borel $\sigma$-algebra with respect to $\Prb$.  Here $M_F\in\R$ is an arbitrary constant.

We treat the first argument of $F$ as time and the second one as space.
We write the time argument of $F$  as a subscript obtaining $F_k:\R\to [M_F,+\infty)$, $k\in\Z$. We introduce the space-time shifts $(\theta^{n,x})_{n\in \Z, x\in \R}$ acting on $\Omega$, defined by $\theta^{n,x}F_k(y) = F_{k+n}(y+x)$, and we assume that these space-time shifts preserve~$\Prb$ so that $F$ is space-time stationary. We assume the collection $(F_k)_{k\in \Z}$ to be independent.

We also require 
\begin{equation}\label{potentialIntegrabilityZeroTemp}
    \E\, F_0(0) < \infty.
\end{equation}
As a result, we work with a large class of random potentials: bounded from below finite mean stationary processes, i.i.d.\ in time. 

Now let us describe our requirements 
on the energy $V:\R\to\R$. We assume that it belongs to $C^2(\R)$ and satisfies
\begin{equation}\label{kineticEnergyGrowthAtInfinity}
    \lim_{|x|\to \infty} V(x) = +\infty
\end{equation}
and
\begin{equation}\label{secondDerivReq}
     \limsup_{|x|\to \infty}\frac{V''(x)}{V(x)} < \infty.
\end{equation}
The coercivity requirement~\eqref{kineticEnergyGrowthAtInfinity} ensures that
\begin{equation}\notag
M_V:=\min_{x\in\R} V(x)>-\infty,  
\end{equation} 
which implies that
there exists a minimizer for the zero temperature problem defined below in \eqref{eq:AF_def}. Assumption \eqref{secondDerivReq} restricts the growth and
oscillation of $V$ at infinity. In the positive temperature setting,  we impose additional conditions on $V$, see \Cref{directedPolymerModelSetUp}.
Note that we do not require $V$ to be convex or monotone.  Moreover, we do not require that $p(x)$ obtained from $V(x)$ via 
\begin{equation}
\label{eq:p-from-V}
p(x)=e^{-V(x)},\quad x\in\R,
\end{equation}
(see \eqref{eq:V-from-p}) is a probability density. Note that polynomial energies $V$ defined in \eqref{eq:exampleKineticEnergies} satisfy our requirements.

For a path $\gamma:\{0,\dots,n\}\to \R$, we use 
\eqref{eq:path_energy} to 
define $A^n(\gamma)=A^{0,n}(\gamma)$. 
We define the set of admissible paths of length~$n$ with slope $v$ to be
\begin{equation}
\label{eq:Gamma_nv}
\Gamma^n(v) := \{\gamma\in \R^{n+1}\,:\,\gamma_0 = 0,\,\gamma_{n} = vn\}
\end{equation}
and consider the minimization problem 
\begin{equation}\label{eq:AF_def}
    A^n_*(v) = \inf\Big\{A^n(\gamma)\,:\,\gamma\in \Gamma^n(v)\Big\}.
\end{equation}

The following result claims that the shape function is well-defined and convex
under our assumptions. Although the proof follows standard lines,  we prove this theorem in Section~\ref{shapeTheoremSection} for completeness.
\begin{theorem}\label{shapeTheorem}
    Under the above assumptions, there is a convex deterministic function $\Lambda_0:\R\to \R$ such that for every $v\in \R,$
    \begin{equation}\label{eq:shapeFunction}
        \Lambda_0(v) = \lim_{n\to \infty} \frac{1}{n}A^n_*(v)
    \end{equation}
    $\Prb$-almost surely.
\end{theorem}

The next theorem is a precise statement of our main result, Theorem~\ref{th:main_high_level},  on differentiability of the shape function
 in the zero temperature case. We prove it in Section~\ref{differentiabilityTheoremProofs}. We denote by $\gamma_A^n$ the minimizer to \eqref{eq:AF_def} given in \Cref{lem:existenceOfMinimizer} in \Cref{sec:shearedEnvironmentZeroTempIntro}.
\begin{theorem}\label{thm:differentiabilityZeroTemp}
 Under the same assumptions,   the function $\Lambda_0$ is differentiable everywhere. Furthermore, for every $v\in \R$,
    \begin{equation}\label{shapeFcnDerivative}
        \Lambda_0'(v) = \lim_{n\to \infty}\frac{1}{n}\sum_{k=0}^{n-1}V'(\Delta_k \gamma_A^n(v))
    \end{equation}
    $\Prb$-almost surely.
\end{theorem}

\begin{remark}
For strictly convex $V$,  
\Cref{thm:differentiabilityZeroTemp} provides heuristic evidence for strict convexity of $\Lambda_0$.
One might expect that 
for $w>v$, the limiting statistics of the increments $(\Delta_k\gamma^n(v))_{k=1}^n$ are dominated in some sense by those of $(\Delta_k\gamma^n(w))_{k=1}^n$. If $V'$ is strictly increasing, this would suggest that 
the right-hand side of 
\eqref{shapeFcnDerivative} is increasing in $v$. 
\end{remark}

\subsection{Positive temperature setting and main results}\label{directedPolymerModelSetUp}
The setup for positive temperature directed polymers is similar to the zero temperature case, but we
will need to impose additional mild restrictions on $F$ and~$V$.
We work on the same probability space as in \Cref{sec:zeroTempModelSetUp}
and use the notation \eqref{eq:p2p_meas}--\eqref{eq:p2p_pf} from 
Section~\ref{sec:intro-poly}.

Our conditions on $F$ and $V$ guarantee that $0<Z_{x,y}^{m,n}<\infty$ for all $x,y\in \R$, $m,n\in\Z$ satisfying $m\le n$ and so all $\mu_{x,y}^{m,n}$ are well-defined.
Although we do not require that $p(\cdot)$ defined by 
\eqref{eq:p-from-V} is a probability density, $\mu_{x,y}^{m,n}$ are probability
measures due to the normalization by $Z_{x,y}^{m,n}$.

We strengthen  the requirement \eqref{potentialIntegrabilityZeroTemp} on $F$ and assume that
\begin{equation}\notag
    \E\,F_k^*(0) < \infty,
\end{equation}
where
\begin{equation}\label{Fkmax}
    F_k^*(x) = \sup_{|y-x|\le 1/2}F_k(y).
\end{equation}

We also introduce two additional requirements on $V$. We assume that 
\begin{equation}\label{linearGrowthV}
    \liminf_{|x|\to \infty}|V'(x)| > 0
\end{equation}
and that there is $\theta\in (0,1)$ such that
\begin{equation}\label{growthConstraintV}
    \limsup_{|x|\to \infty}\frac{|V'(x)|}{|V(x)|^{\theta}} < \infty.
\end{equation}
We still do not require $p$ given by \eqref{eq:p-from-V} to
be a probability density but \eqref{linearGrowthV} implies 
monotonicity and at least linear growth of $V$ at infinity, so 
$p\in L^1(\R)$. 

Polynomial energies $V$ defined in \eqref{eq:exampleKineticEnergies} still satisfy our requirements. In addition, our requirements are preserved under a large class of $C^2$ additive perturbations including perturbations vanishing at infinity with their first and second derivatives.

All of our assumptions on $F$ and $V$ in both settings are preserved under affine transformations of the form $F\mapsto \beta F + C$ and $V\mapsto \alpha V + c$ with constants $\alpha,\beta>0$ and $c,C\in\R$, up to inconsequential adjustments of $M_F$ and~$M_V$.

Let us define $Z^n(v)=Z^{0,n}_{0,vn}$ and 
$\mu_{v}^{n}=\mu_{0,vn}^{0,n}$. 
We will often write   $\mu_{v}^n(f(\gamma))$ to denote 
the expectation of a function $f:\R^{n+1}\to \R$ with respect to $\mu_{v}^n.$

The positive temperature counterpart to~$A^n_*(v)$ is 
\[-\log Z^n(v),\]
the
finite volume free energy in direction $v$.

The basic shape function theorem is proved in Section~\ref{shapeTheoremSection} for completeness:

\begin{theorem}\label{shapeTheoremPolymer}
    Under the above assumptions, there is a convex deterministic function $\Lambda_1:\R\to \R$ such that for every $v\in \R,$
    \begin{equation}\notag
        \Lambda_1(v) = -\lim_{n\to \infty} \frac{1}{n}\log Z^n(v)
    \end{equation}
    $\Prb$-almost surely.
\end{theorem}

The next theorem is a precise statement of our main result, Theorem~\ref{th:main_high_level},  on differentiability of the shape function
 in the positive temperature case. We prove it in Section~\ref{differentiabilityTheoremProofs}.
\begin{theorem}\label{thm:differentiabilityPolymer} Under the same assumptions, the function $\Lambda_1$ is differentiable everywhere. Furthermore, for every $v\in \R$
    \begin{equation}\notag
        \Lambda_1'(v) = \lim_{n\to \infty}\mu_{v}^n\bigg(\frac{1}{n}\sum_{k=0}^{n-1}V'(\Delta_k \gamma)\bigg)
    \end{equation}
    $\Prb$-almost surely.
\end{theorem}

\begin{remark} For strictly convex $V$, similarly to the zero temperature case,   our  formula for $ \Lambda_1'(v)$ provides heuristic evidence
of strict convexity of $\Lambda_1$.
\end{remark}

The layout of the remainder of the paper is as follows. 
 In \Cref{shearedEnvironmentIntroduction}, we introduce the sheared environment, a key ingredient in our 
 arguments. In \Cref{differentiabilityTheoremProofs}, we 
 prove our main results, \Cref{thm:differentiabilityZeroTemp,thm:differentiabilityPolymer}.
Sections~\ref{shapeTheoremSection} and~\ref{auxiliaryProofs} play the role of appendices:
 we give fairly standard proofs of  \Cref{shapeTheorem,shapeTheoremPolymer} in \Cref{shapeTheoremSection}, and 
 we prove an auxiliary technical lemma  in \Cref{auxiliaryProofs}.

\section{The Sheared Environment}\label{shearedEnvironmentIntroduction}

A key ingredient to our proofs of \Cref{thm:differentiabilityZeroTemp,thm:differentiabilityPolymer} is the \textit{sheared environment}. Our assumptions on the environment guarantee that under shear transformations $\Prb$ is preserved 
and~$V$ is shifted. Thus, proving $\Prb$-almost sure statements in the original environment is equivalent to proving the analogous $\Prb$-almost sure statement in the sheared environment for shifted $V$. In this section we introduce the relevant concepts and notation.

The shear maps $\Xi_v$, $v\in\R$, acting on paths are defined by 
\[
    (\Xi_v \gamma)_k = \gamma + vk.
\]
For $v\in\R$, the shear map acting on the environment, $\Xi_v^*:\Omega\to \Omega$, is given by
\[(\Xi_v^* F)_k (x)= F_k(x  + vk),\quad x\in\R.
\]
These definitions ensure
\begin{align}
\label{eq:shears-agree}
F_k((\Xi_v\gamma)_k)&= (\Xi^*_v F)_k(\gamma_k).
\end{align}

Our stationarity and independence assumptions on the environment imply that $\Xi_v^*$ is measure preserving
for every $v\in\R$:
\begin{equation}\label{FShearInvariant}
    \Xi_v^* F \stackrel{d}{=} F.
\end{equation}

\subsection{Applying sheared environment to the zero temperature model}\label{sec:shearedEnvironmentZeroTempIntro}
In order to prove \Cref{thm:differentiabilityZeroTemp}, we will study problem \eqref{eq:AF_def} in a sheared environment. Recalling the definition of $\Gamma^n(v)$ in~\eqref{eq:Gamma_nv}, we define
\begin{align}\label{BF_def}
    B_*^n(v) = \inf\Big\{B^n(v,\gamma)\,:\, \gamma\in\Gamma^n(0)\Big\},
\end{align}
where
\[B^n(v,\gamma) =  \sum_{k=0}^{n-1} V(\Delta_k \gamma + v) +  \sum_{k=0}^{n-1}F_k(\gamma_k).\]
Note that for $\gamma\in\Gamma^n(0)$, we can define  $\gamma'=\Xi_v\gamma\in \Gamma^n(v)$, and write
\[
B^n(v,\gamma) =  \sum_{k=0}^{n-1} V(\Delta_k\gamma')+\sum_{k=0}^{n-1}\Xi^*_{-v} F_k(\gamma').
\]
Hence, $B_*^n(v)$ for an environment realization $F$ equals $A_*^n(v)$ for the sheared environment $\Xi^*_{-v} F$. The latter has the same distribution as $F$, so, for a fixed $v\in \R,$ all distributional results for $A_*^n(v)$ and $B_*^n(v)$ are equivalent.  However,
an important benefit of $B_*^n(v)$ is that the set of admissible paths~$\Gamma^n(0)$ in the definition \eqref{BF_def}
does not depend on $v$. This allows for useful direct comparisons between $B_*^n(v)$ and $B_*^n(w)$ for $v\neq w$,
that  are not available for~$A_*^n$, see its definition in \eqref{eq:AF_def}. 

The following lemma states existence of minimizers realizing the infima in the  definitions of $A_*^n(v)$ and $B_*^n(v)$
as measurable maps and makes the distributional relation between $A^n_*(v)$ and $B^n_*(v)$  more precise. 
\begin{lemma}\label{lem:existenceOfMinimizer}
    For all realizations of $F$, all $n\in\N$, all $v\in\R$, there exist minimizers to \eqref{eq:AF_def} and \eqref{BF_def}. Furthermore, there are measurable maps
    \[\gamma_A^n(v):\Omega\to \Gamma^n(v),\]
    \[\gamma_B^n(v):\Omega\to \Gamma^n(0)\]
    providing the infima
    in \eqref{eq:AF_def} and \eqref{BF_def}, respectively, and such that 
    \begin{align}\label{shearedEnvironmentEquality}
        (B_*^n(v),\Xi_v\gamma^n_B(v))_{n\in \N} \stackrel{d}{=} (A_*^n(v),\gamma^n_A(v))_{n\in \N}.
    \end{align}
\end{lemma}
See \Cref{auxiliaryProofs} for the proof of \Cref{lem:existenceOfMinimizer}.
Just as for $A_*^n(v),B_*^n(v)$, we will also omit the dependence of $\gamma_A^n(v),$ and $\gamma_B^n(v)$ on the random environment. 

\begin{remark}\label{shapeTheoremBn}
     \Cref{shapeTheorem} and Lemma~\ref{lem:existenceOfMinimizer} imply that for every $v\in \R,$
    \begin{equation}
    \label{eq:shapeTheoremBn}
    \Lambda_0(v) = \lim_{n\to \infty}\frac{1}{n}B_*^n(v),\quad \Prb\text{-a.s.}
    \end{equation}
    This formula is our starting point in the proof of differentiability of $\Lambda_0$. 
\end{remark}

\subsection{Applying sheared environment to the positive temperature model}

We define the polymer measure in the sheared environment with slope $v$ to be the measure absolutely continuous with respect 
to~$\refpm_{0,0}^{0,n}(d\gamma)$ given by
\begin{equation}\label{eq:polymerDef}
    \mushear_v^n(d\gamma_0,\dots, d\gamma_n) = \frac{1}{\Zshear^n(v)}
      e^{- B^n(v,\gamma)} \refpm_{0,0}^{0,n}(d\gamma).
\end{equation}
Here, $\Zshear^n(v)$ is the partition function defined so that $\mushear_v^n$ is a probability measure:
\begin{equation}
    \Zshear^n(v) 
    =
     \int_{\R^{n+1}}e^{- B^n(v,\gamma)} \refpm_{0,0}^{0,n}(d\gamma)=
     \int_{\R^{n+1}}e^{- \sum_{k=0}^{n-1}V(\Delta_k \gamma + v)- \sum_{k=0}^{n-1}F_k(\gamma_k) } \refpm_{0,0}^{0,n}(d\gamma).
    \label{eq:Zshear}
\end{equation}
We will sometimes write $\dx{\gamma}$ instead of $\refpm_{0,0}^{0,n}(d\gamma)$ with the understanding that the starting and ending points are fixed in the sheared model. We will often write~$\mushear_v^n(f(\gamma))$ to denote the expectation of a function $f:\R^{n+1}\to \R$ with respect to $\mushear_v^n.$

\begin{lemma}
    The following equality in distribution holds:
    \begin{equation}\label{equalityInDistributionPolymer}
        (\Zshear^n(v), \mushear_v^n\Xi_v^{-1})_{n\in \N} \stackrel{d}{=} (Z^n(v), \mu_{0,vn}^n)_{n\in \N}.
    \end{equation}
\end{lemma}
\begin{proof}
    Changing variables $\gamma'=\Xi_v \gamma$, or $\gamma=\Xi^{-1}_{v} \gamma'=\Xi_{-v} \gamma'$,  in~\eqref{eq:Zshear}
    and \eqref{eq:polymerDef}, we obtain using~\eqref{eq:shears-agree}:
    \begin{align*}
    \Zshear^n(v) = &
\int_{\R^n} e^{ - \sum_{k=0}^{n-1}V(\Delta_k \gamma')-\sum_{k=0}^{n-1}F_k((\Xi_{-v} \gamma')_k)}\refpm_{0,vn}^{0,n}(d\gamma')
\\
=&\int_{\R^n} e^{ - \sum_{k=0}^{n-1}V(\Delta_k \gamma')-\sum_{k=0}^{n-1}\Xi^*_{-v}F_k(\gamma'_k)}\refpm_{0,vn}^{0,n}(d\gamma')
    \end{align*}
    and
    \begin{align*}
        \mushear_v^n\Xi_{v}^{-1}(d\gamma') &= \frac{1}{\Zshear^n(v)} e^{- \sum_{k=0}^{n-1}V(\Delta_k \gamma')-\sum_{k=0}^{n-1}F_k((\Xi_{-v} \gamma')_k) }\refpm_{0,vn}^{0,n}(d\gamma')
        \\
        &=
         \frac{1}{\Zshear^n(v)} e^{- \sum_{k=0}^{n-1}V(\Delta_k \gamma')-\sum_{k=0}^{n-1}(\Xi^*_{-v}F)_k( \gamma'_k) }\refpm_{0,vn}^{0,n}(d\gamma')
    \end{align*}
 Applying the equality in distribution  \eqref{FShearInvariant} to these displays gives us \eqref{equalityInDistributionPolymer}.
\end{proof}

\begin{remark}\label{shapeTheoremShear}
     \Cref{shapeTheorem} and Lemma~\ref{equalityInDistributionPolymer}  imply that for every $v\in \R,$
    \begin{equation}\label{eq:shapeTheoremBnPolymer}
    \Lambda_1(v) = -\lim_{n\to \infty}\frac{1}{n}\ln \tilde Z^n(v),\quad \Prb\text{-a.s.}
    \end{equation}
    This formula is our starting point in the proof of differentiability of $\Lambda_1$. 
\end{remark}

\section{Proofs of \Cref{thm:differentiabilityZeroTemp} and \Cref{thm:differentiabilityPolymer}}\label{differentiabilityTheoremProofs}

We begin the proofs of \Cref{thm:differentiabilityZeroTemp,thm:differentiabilityPolymer} with useful deterministic lemmas.
We prove \Cref{thm:differentiabilityZeroTemp} in \Cref{zeroTempDifferentiabilitySection} and \Cref{thm:differentiabilityPolymer} 
in \Cref{polymerDifferentiabilitySection}.

\subsection{Deterministic lemmas}
First, we review some basic facts concerning convex functions, see, e.g.,  Section 1.4 of \cite{lars-erik_2019}. For a convex or concave function $f$ defined in an open set $\mathcal{O}\subset \R,$ the left- and right-hand derivatives of $f$ exist 
at all  $x\in \mathcal{O}$. We 
denote them, respectively, by 
\[\partial^-f(x),\partial^+ f(x).\]
If $f$ is convex and $x< y$ then we have
\[\partial^- f(x) \le \partial^+ f(x) \le \partial^-f(y) \le \partial^+f(y)\]
with the inequalities reversed for concave $f$. The equality $\partial^- f(x) = \partial^+f(x)$ holds if and only if~$f$ is differentiable at $x$.

If a  function of several variables is convex or concave with respect to some variable, then for the associated  left- or right-hand derivatives,
we include that variable in the subscript of $\partial^-$ or $\partial^+$.

The following lemma is 
at the heart of our argument  for the main results. It  derives 
 differentiability of a convex function $f$ at a point~$x_0$ from an approximate linear domination condition.
It also gives a formula 
for $f'(x_0)$ in terms of approximations to~$f$.  
In Sections~\ref{zeroTempDifferentiabilitySection} and~\ref{polymerDifferentiabilitySection},
we will show that our approximate linear domination condition holds for
the sequences $(\frac{1}{n}B^n_*)_{n\in \N}$ and $(-\log \Zshear^n)_{n\in \N}$ approximating the
shape functions~$\Lambda_0$ and~$\Lambda_1$.

\begin{lemma}\label{deterministicDifferentiability}
    Let $\mathcal{O}\subset \R$ be an open set, $x_0\in \mathcal{O}$, and $\mathcal{D}\subset \mathcal{O}$ be dense in $\mathcal{O}.$ Let $(f_n)_{n\in \N}$ be a sequence of functions from $\mathcal{O}$ to $\R$ and $f:\mathcal{O}\to \R$ be a function such that for all $x\in \mathcal{D}\cup\{x_0\}$,
    \[\lim_{n\to \infty}f_n(x) = f(x).\]
    Suppose also that there exists a sequence of real numbers $(g_n)_{n\in \N}$ and a function $h:\mathcal{O}\to \R$ such that the following holds:
    \begin{enumerate}
        \item\label{linearDomination1} There is $\delta > 0$ such that for all $x\in \mathcal{D}\cap (x_0-\delta,x_0+\delta)$ and $n\in \N,$ 
        \begin{equation}\label{finiteDerivativeInequality}
            f_n(x) - f_n(x_0) \le (x-x_0)g_n + h(x),
        \end{equation}
        \item\label{linearDomination2} $\lim_{x\to x_0}\frac{h(x)}{x-x_0 } = 0$.
    \end{enumerate}
    If $f$ is convex, then: $f$ is differentiable at $x_0$, the sequence $(g_n)_{n\in \N}$ converges, and 
    \begin{equation}\label{convexImpliesDifferentiable}
        f'(x_0) = \lim_{n\to \infty}g_n.
    \end{equation}
    If $f$ is concave, then
    \begin{equation}\label{concaveImpliesBounds}
        \partial^+ f(x_0) \le \liminf_{n\to \infty} g_n \le \limsup_{n\to \infty} g_n\le \partial^- f(x_0).
    \end{equation}
\end{lemma}
\begin{proof}[Proof of \Cref{deterministicDifferentiability}]
    Let $(x_m)_{m\in \N}\subset \mathcal{D}\cap (x_0,x_0+\delta)$ be a sequence such that $x_m\searrow x_0$. Inequality \eqref{finiteDerivativeInequality} along with the inequality $x_m> x_0$ implies 
    \begin{equation}\notag
        \frac{f_n(x_m) - f_n(x_0)}{x_m-x_0} \le g_n + \frac{h(x_m)}{x_m-x_0}.
    \end{equation}
    Taking $n\to \infty$ in the above we obtain
    \begin{equation}\notag
        \frac{f(x_m) - f(x_0)}{x_m - x_0} \le \liminf_{n\to \infty} g_n + \frac{h(x_m)}{x_m-x_0}.
    \end{equation}
    If $f$ is either concave or convex, then, taking $m\to \infty$, we obtain 
    \begin{equation}\label{rightDerivUpperBound}
        \partial^+ f(x_0) \le \liminf_{n\to \infty} g_n.
    \end{equation}
    Now take a sequence $(y_m)_{m\in \N}\subset \mathcal{D}\cap (x_0-\delta,x_0)$ such that $y_m\nearrow x_0$. 
    Inequality~\eqref{finiteDerivativeInequality} implies 
    \begin{equation}
        \frac{f_n(y_m) - f_n(x_0)}{y_m-x_0} \ge g_n + \frac{h(y_m)}{y_m-x_0}.\notag
    \end{equation}
    Taking $n\to \infty$ we obtain
    \begin{equation}
        \frac{f(y_m) - f(x_0)}{y_m - x_0} \ge \limsup_{n\to \infty} g_n + \frac{h(y_m)}{y_m-x_0}.\notag
    \end{equation}
    If $f$ is either convex or concave, then, taking $m\to \infty$, we obtain 
    \begin{equation}\label{leftDerivLowerBound}
        \partial^- f(x_0) \ge \limsup_{n\to \infty} g_n.
    \end{equation}
     If $f$ is concave, inequalities \eqref{rightDerivUpperBound} and \eqref{leftDerivLowerBound} imply \eqref{concaveImpliesBounds}. Now assume that $f$ is convex. Then  \eqref{rightDerivUpperBound} and \eqref{leftDerivLowerBound} along with the inequality 
    \[\partial^-f(x_0) \le \partial^+ f(x_0)\]
    imply that 
    \[\partial^- f(x_0) = \partial^+ f(x_0)\]
    and 
    \[\limsup_{n\to \infty} g_n = \liminf_{n\to \infty}g_n.\]
    The above two inequalities imply the result because differentiability of a convex function is equivalent to the agreement of its left- and right-hand derivatives.
\end{proof}

The following technical lemma shows that 
requirements \eqref{kineticEnergyGrowthAtInfinity} and \eqref{secondDerivReq}  imply a stronger version of \eqref{secondDerivReq}.
\begin{lemma}\label{secondDerivReqImprovement}
    If \eqref{kineticEnergyGrowthAtInfinity} and \eqref{secondDerivReq} hold, then 
    \begin{equation}\label{secondDerivMaxBound}
        \limsup_{|x|\to \infty}\sup_{|r|\le 1}\frac{V''(x+r)}{V(x)} < \infty
    \end{equation}
    and 
    \begin{equation}\label{derivativeRatio}
        \limsup_{|x|\to \infty}\Big|\frac{V'(x)}{V(x)}\Big| < \infty.
    \end{equation}
\end{lemma}

\begin{proof}
    First, we prove the version of \eqref{derivativeRatio} where the $\limsup$ is taken as $x\to+\infty.$ For sufficiently large $x$, $V(x)\ne 0$, and we can define $g(x) = \frac{V'(x)}{V(x)}$,
  so that $\frac{V''(x)}{V(x)} = g'(x) + |g(x)|^2.$  Display \eqref{secondDerivReq} means that
    \[\limsup_{x\to \infty}\Big[g'(x) + |g(x)|^2\Big] < \infty.\]
    Let $L,C > 0$ be such that 
    \[\sup_{x>L}\{g'(x) + |g(x)|^2\} < C.\]
    If $x\in \R$ satisfies $x>L$ and $|g(x)| > \sqrt{C},$ the above implies that $g'(x) < 0$ and so~$g$ is decreasing at such $x$. It follows that 
    \[\sup_{x>L}|g(x)| \le \max(g(L),\sqrt{C}) < \infty.\]
    This proves \eqref{derivativeRatio} for $x\to +\infty.$ The case $x\to -\infty$ is treated similarly.

    Let us prove \eqref{secondDerivMaxBound}. Display \eqref{derivativeRatio} implies that there is $R>0$ and $c > 0$ such that if $|x|>R$ then $V'(x) \le c V(x)$ and $-V'(x) \le c V(x).$ It follows from Gronwall's inequality that if $|x|,|x+r|>R$, then $V(x+r)\le e^{c|r|}V(x).$ Thus,
    \[ \limsup_{|x|\to \infty}\sup_{|r|\le 1}\frac{V''(x+r)}{V(x)} \le \limsup_{|x|\to \infty}\sup_{|r|\le 1}\frac{V''(x+r)}{V(x+r)} \cdot \sup_{|r|\le 1}\frac{V(x+r)}{V(x)}< \infty,\]
    proving \eqref{secondDerivMaxBound}.
\end{proof}

\subsection{Zero temperature model}\label{zeroTempDifferentiabilitySection}
For perturbative analysis, it will be helpful to extend the definition of the zero temperature model in \eqref{BF_def} to include scalar multiples of the environment and the kinetic action by parameters $\alpha,\beta> 0$. Specifically, let
\[B^n(v,\gamma,\alpha,\beta) =\alpha \sum_{k=0}^{n-1}V(\Delta_k \gamma + v)+\beta\sum_{k=0}^{n-1}F_k(\gamma_k)\]
and
\begin{align}\notag
    B_*^n(v,\alpha,\beta) = \inf\Big\{B^n(v,\gamma,\alpha,\beta)\,:\, \gamma\in\Gamma^n(0)\Big\}.
\end{align}
If $F$ and $V$ obey the assumptions outlined in \Cref{sec:zeroTempModelSetUp}, then  so do $\alpha V$ and $\beta F$. Thus, \Cref{shapeTheorem} holds if we replace~$B^n_*(v)$ by $B^n_*(v,\alpha,\beta)$, and allows to define the associated shape function 
\[
\Lambda_0(v,\alpha,\beta)=\lim_{n\to\infty}\frac{1}{n}B_*^n(v,\alpha,\beta) 
\]   for all $\alpha,\beta> 0$.

The proposition below establishes concavity of $\Lambda_0$ with respect to $\alpha$ and $\beta.$ Other types of concavity results have been noted previously in related zero temperature models, see Lemma 3.1 in \cite{batesEmpiricalEnvironment} or Section 6.5 of \cite{Hammersley1965}. Although in the proof of our main differentiability theorem we only use the upper bound contained in~\eqref{empiricalKineticEnergy}, proving the full proposition requires little extra work.

We need some additional notation.
For any path $\gamma\in \Gamma^n(0)$, we define 
\begin{equation}
\label{eq:Vbar}
\overline{V_n}(v,\gamma) = \frac{1}{n}\sum_{k=0}^{n-1}V(\Delta_k\gamma + v)
\end{equation}
and 
\begin{equation}
\label{eq:Fbar}
\overline{F_n}(\gamma) = \frac{1}{n}\sum_{k=0}^{n-1}F_k(\gamma_k).
\end{equation}
The minimizer for $B^n_*(v,\alpha,\beta)$ constructed in~Lemma~\ref{lem:existenceOfMinimizer} is denoted by $\gamma^n(v,\alpha,\beta)$.

\begin{proposition}\label{concavityOfShapeFcn}
    The shape function $\Lambda_0$ is concave in $\alpha$ and $\beta.$ Furthermore, for every $v\in \R,$ $\alpha,\beta>0$,
    \begin{multline}\label{empiricalKineticEnergy}
        \partial_\alpha^+\Lambda_0(v,\alpha,\beta) \le \liminf_{n\to \infty} \overline{V_n}(v,\gamma^n(v,\alpha,\beta)) 
       \\ \le \limsup_{n\to \infty}\overline{V_n}(v,\gamma^n(v,\alpha,\beta))) \le\partial_\alpha^- \Lambda_0(v,\alpha,\beta)
    \end{multline} 
    and
    \begin{multline}\label{empiricalPotentialEnergy}
        \partial_\beta^+\Lambda_0(v,\alpha,\beta) \le \liminf_{n\to \infty} \overline{F_n}(\gamma^n(v,\alpha,\beta)) 
        \\ \le \limsup_{n\to \infty}\overline{F_n}(\gamma^n(v,\alpha,\beta)) \le\partial_\beta^- \Lambda_0(v,\alpha,\beta)
    \end{multline}
    with $\Prb$-probability one.
\end{proposition}
\begin{proof}
    Let $\alpha_1,\alpha_2,\beta_1,\beta_2> 0$ and $t\in [0,1]$. Let $\alpha = t\alpha_1 + (1-t)\alpha_2$ and $\beta = t\beta_1 + (1-t)\beta_2.$ Then,
    \begin{align*}
        B_*^n(v,\alpha,\beta) &= \min\Big\{  \alpha \sum_{k=0}^{n-1}V(\Delta_k \gamma+v) + \beta \sum_{k=0}^{n-1} F_k(\gamma_k)\,:\,\gamma\in\Gamma^n(0)\Big\}\\
        & \ge t \min\Big\{  \alpha_1 \sum_{k=0}^{n-1}V(\Delta_k \gamma+v) + \beta_1 \sum_{k=0}^{n-1} F_k(\gamma_k)\,:\,\gamma\in\Gamma^n(0)\Big\} \\
        & \qquad\qquad + (1-t) \min\Big\{  \alpha_2 \sum_{k=0}^{n-1}V(\Delta_k \gamma+v) + \beta_2 \sum_{k=0}^{n-1} F_k(\gamma_k)\,:\,\gamma\in\Gamma^n(0)\Big\}\\
        & = tB_*^n(v,\alpha_1,\beta_1) + (1-t)B_*^n(v,\alpha_2,\beta_2),
    \end{align*}
    which proves concavity of 
    $B_*^n(v,\alpha,\beta)$ in $\alpha$ and $\beta$.    
 \Cref{shapeTheorem} then implies that $\Lambda_0$ is concave in $(\alpha,\beta)$ as well since
    \begin{align}\label{convexityInLimit}
        \Lambda_0(v,\alpha,\beta) &= \lim_{n\to \infty}B_*^n(v,\alpha,\beta)\notag \\
        & \ge t\lim_{n\to \infty}B_*^n(v,\alpha_1,\beta_1) + (1-t)\lim_{n\to \infty}B_*^n(v,\alpha_2,\beta_2)\notag\\
        & = t\Lambda_0(v,\alpha_1,\beta_1) + (1-t)\Lambda_0(v,\alpha_2,\beta_2).
    \end{align}
    Now we establish \eqref{empiricalKineticEnergy}. For $\beta,\alpha,\alpha'> 0$,
    we can use the minimizer  $\gamma^n(v,\alpha,\beta)$ realizing $B_*^n(v,\alpha,\beta)$ to estimate  
    $B_*^n(v,\alpha',\beta)$:
    \begin{align*}
        B_*^n(v,\alpha',\beta) &\le \beta\sum_{k=0}^{n-1}F_k(\gamma^n(v,\alpha,\beta)) + \alpha'\sum_{k=0}^{n-1}V(\Delta_k \gamma^n(v,\alpha,\beta)+v)\\
        & = B_*^n(v,\alpha,\beta) + (\alpha'-\alpha) \sum_{k=0}^{n-1}V(\Delta_k \gamma^n(v,\alpha,\beta)+v)\\
        & = B_*^n(v,\alpha,\beta) + 
        (\alpha'-\alpha) n\overline{V_n}(v,\gamma^n(v,\alpha,\beta)).
    \end{align*}
    This estimate allows us to apply \Cref{deterministicDifferentiability}, and specifically \eqref{concaveImpliesBounds} to conclude~\eqref{empiricalKineticEnergy}. We can choose $\mathcal{D}$ to be any countable dense subset of $\R.$ Fix some $\alpha,\beta>0.$ Due to \Cref{shapeTheorem}, with $\Prb$-probability one, for all $x\in \mathcal{D}\cup\{\alpha\},$
    \[\lim_{n\to \infty}\frac{1}{n}B_*^n(v,x,\beta) = \Lambda_0(v,x,\beta).\]
    Thus, with $\Prb$-probability one, the conditions of \Cref{deterministicDifferentiability} are satisfied with
    \begin{align*}
        x_0 &= \alpha,\\
        f_n(x) &= \frac{1}{n}B_*^n(v,x,\beta),\\
        f(x) &= \Lambda_0(v,x,\beta),\\
        g_n &= \overline{V_n}(v,\gamma^n(v,\alpha,\beta)),\\
        h(x) &= 0,
    \end{align*}
    and~\eqref{empiricalKineticEnergy} follows.
    The estimates in \eqref{empiricalPotentialEnergy} follow similarly.
\end{proof}

\begin{remark}
    Similarly to the Proposition 2.3 in \cite{Com17} or the stronger results contained in \cite{batesEmpiricalEnvironment}, \Cref{concavityOfShapeFcn} and the fact that convex functions are differentiable Lebesgue almost everywhere imply that for Lebesgue almost every $(\alpha,\beta)\in (0,\infty)^2$, $\overline{V_n}(v,\gamma^n(v,\alpha,\beta)))$ and $\overline{F_n}(\gamma^n(v,\alpha,\beta))$ converge $\Prb$-almost surely. An analogous result holds in the positive temperature case, see \Cref{concavityOfShapeFcnPolymer}.
\end{remark}

For $v\in\R$, we denote
   \begin{equation*}
        M_{v,\infty}=\limsup_{n\to \infty}\frac{1}{n}\sum_{k=0}^{n-1} \sup_{|r|\le 1}V''(\Delta_k\gamma^n(v) + v + r),
    \end{equation*}
where
\begin{equation}
\label{eq:gammanv}
\gamma^n(v)=\gamma^{0,n}_B(v)=\gamma^{n}(v,1,1).
\end{equation}

\begin{lemma}\label{boundedSecondDerivSums}
   For all $v\in \R,$ 
    \begin{equation}\label{limsupSecondDerivative}
       M_{v,\infty}< \infty
    \end{equation}
    with $\Prb$-probability one.
\end{lemma}
\begin{proof} Let us take any path $\gamma\in \Gamma^n(0)$.
    According to \Cref{secondDerivReqImprovement}, there is $L>0$ and $C>0$ such that if $|\Delta_k \gamma^n| > L$, then 
    \[\sup_{|r|\le 1}V''(\Delta_k\gamma + v + r) \le C V(\Delta_k\gamma + v).\]
    Define $A := \sup_{|x|\le L,|r|<1} V''(x + v + r) < \infty$ and $b = \min(\inf_{x\in \R} V(x),0)$. We have
    \begin{align}\label{upperBoundL}
        \frac{1}{n}&\sum_{k=0}^{n-1} \sup_{|r|\le 1}V''(\Delta_k\gamma + v + r)\notag\\
         & = \frac{1}{n}\sum_{k=0}^{n-1} \sup_{|r|\le 1}V''(\Delta_k\gamma + v + r) \1_{|\Delta_k\gamma| \le L} + \frac{1}{n}\sum_{k=0}^{n-1} \sup_{|r|\le 1}V''(\Delta_k\gamma + v + r) \1_{|\Delta_k\gamma| > L}\notag\\
         & \le \frac{1}{n}\sum_{k=0}^{n-1} \sup_{|r|\le 1}V''(\Delta_k\gamma + v + r) \1_{|\Delta_k\gamma| \le L} + \frac{1}{n}\sum_{k=0}^{n-1} C( V(\Delta_k\gamma + v)-b) \1_{|\Delta_k\gamma| > L}\notag\\
        & \le A + \frac{C}{n}\sum_{k=0}^{n-1}V(\Delta_k\gamma + v) -C b\notag\\
        & = A + C\overline{V_n}(v,\gamma) - C b.
    \end{align}
    Substituting $\gamma=\gamma^n(v)$, taking $\limsup$ of both sides, and applying~\Cref{concavityOfShapeFcn}, we obtain~\eqref{limsupSecondDerivative}. 
\end{proof}

Now we can prove the main differentiability result for the zero temperature model.

\begin{proof}[Proof of \Cref{thm:differentiabilityZeroTemp}] Let us fix any $v\in\R$.
    For all $x,w\in \R$ satisfying $|w-v|\le 1,$ Taylor's Theorem implies
    \begin{align}\label{taylorTheorem}
        V(x + w) \le V(x + v) + (w-v)V'(x + v) + \frac{1}{2}(w-v)^2\sup_{|r|\le 1}V''(x + v + r).
    \end{align}
    It follows that for any path $\gamma\in\Gamma^n(0)$,
    \begin{multline*}
        \frac{1}{n}B^n(w,\gamma) \le \frac{1}{n}B^n(v,\gamma) 
        + (w-v)\frac{1}{n}\sum_{k=0}^{n-1} V'(\Delta_k\gamma + v)\\ + \frac{1}{2n}(w-v)^2 \sum_{k=0}^{n-1}\sup_{|r|\le 1}V''(\Delta_k\gamma + v + r).
    \end{multline*}
    Plugging 
    $\gamma = \gamma^n(v)$ (see \eqref{eq:gammanv}) into the above, and using the bound 
    $B_*^n(w) \le B^n(w,\gamma^n)$ and identity $B_*^n(v) = B^n(v,\gamma^n)$,
    we obtain
    \begin{equation*}
        \frac{1}{n}B_*^n(w) \le \frac{1}{n}B_*^n(v) + (w-v)\overline{V'_n}(v) +  \frac{1}{2n}(w-v)^2 \sum_{k=0}^{n-1}\sup_{|r|\le 1}V''(\Delta_k\gamma^n(v) + v + r),
    \end{equation*}
where 
\[\overline{V'_{n}}(v) = \frac{1}{n}\sum_{k=0}^{n-1}V'(\Delta_k \gamma^n(v) + v).\]
Thus, 
    Lemma~\ref{boundedSecondDerivSums} implies that for sufficiently large $n$,
    \begin{equation}\notag
        \frac{1}{n}B_*^n(w) \le \frac{1}{n}B_*^n(v) + (w-v)\overline{V'_n}(v) +  \frac{1}{2}(w-v)^2 (M_{v,\infty}+1).
    \end{equation}
    Thus, $B_*^n$ satisfies the approximate linear domination condition~\eqref{finiteDerivativeInequality} at $v\in \R$ with 
    \begin{align*}
        x_0 &= v,\\
        f_n(x) &= \frac{1}{n}B_*^n(x),\\
        g_n & = \overline{V'_n}(x),\\
        h(x) &= \frac{1}{2}(x-v)^2 (M_{v,\infty}+1).
    \end{align*}
 Now our theorem follows from  Lemma~\ref{deterministicDifferentiability}. To apply it, it remains to take $f = \Lambda_0$ (which is convex due to 
 \Cref{shapeTheorem}), take any countable dense set $\mathcal{D}\subset \R$, 
 and notice (see Remark~\ref{shapeTheoremBn}) that 
 on an event of probability one, \eqref{eq:shapeTheoremBn} holds
    for all $x\in \mathcal{D}\cup \{v\}$. This completes the proof.   
    \end{proof}

\subsection{Positive temperature model}\label{polymerDifferentiabilitySection}

Similarly to our treatment of the zero temperature model, we will extend the definition of the directed polymer given in \eqref{eq:polymerDef} to energies depending on $\alpha,\beta> 0:$
\begin{equation}\notag
    \mushear_{v,\alpha,\beta}^n(d\gamma_0,\dots,d\gamma_n) = \frac{1}{\Zshear^n(v,\alpha,\beta)}e^{-B^n(v,\gamma,\alpha,\beta)} 
    \refpm^{0,n}_{0,0}(d\gamma),
\end{equation}
where the partition function $\Zshear^n(v,\alpha,\beta)$ is defined to ensure that $\mushear_{v,\alpha,\beta}^n$ is a probability measure. 

Similarly to the zero temperature case, \Cref{shapeTheoremPolymer}
holds if we replace~$B^n_*(v)$ by $B^n_*(v,\alpha,\beta)$, and allows to define the associated shape function 
\[
\Lambda_1(v,\alpha,\beta)=-\lim_{n\to\infty}\frac{1}{n}\log\Zshear^n(v,\alpha,\beta)
\]   for all $\alpha,\beta> 0$.

Let us recall the notation introduced in~\eqref{eq:Vbar} and~\eqref{eq:Fbar} and establish the positive temperature analog to \Cref{concavityOfShapeFcn}. Similarly to \Cref{concavityOfShapeFcn}, for our main differentiability theorem we only really need the upper bound in \eqref{empiricalKineticEnergyPolymer}.
 A similar concavity result in the fully discrete case is Proposition 2.1 in \cite{Com17}. As with \Cref{concavityOfShapeFcn}, the below lemma implies that for Lebesgue almost every $(\alpha,\beta)\in (0,\infty)^2$, $\mushear_{v,\alpha,\beta}^n(\overline{V_n}(v,\gamma))$ and $\mushear_{v,\alpha,\beta}^n(\overline{F_n}(\gamma))$ converge $\Prb$-almost surely.

\begin{proposition}\label{concavityOfShapeFcnPolymer}
    The shape function $\Lambda$ is concave in $\alpha$ and $\beta.$ Furthermore, for every $v\in \R,$ $\alpha,\beta>0$,
    \begin{multline}\label{empiricalKineticEnergyPolymer}
        \partial_\alpha^+\Lambda_1(v,\alpha,\beta) \le \liminf_{n\to \infty} \mushear_{v,\alpha,\beta}^n(\overline{V_n}(v,\gamma))
        \\ \le \limsup_{n\to \infty}\mushear_{v,\alpha,\beta}^n(\overline{V_n}(v,\gamma)) \le\partial_\alpha^- \Lambda_1(v,\alpha,\beta)
    \end{multline} 
    and
    \begin{multline}\label{empiricalPotentialEnergyPolymer}
        \partial_\beta^+\Lambda_1(v,\alpha,\beta) \le \liminf_{n\to \infty} \mushear_{v,\alpha,\beta}^n(\overline{F_n}(\gamma)) 
        \\ \le \limsup_{n\to \infty}\mushear_{v,\alpha,\beta}^n(\overline{F_n}(\gamma)) \le\partial_\beta^- \Lambda_1(v,\alpha,\beta)
    \end{multline}
    $\Prb$-almost surely.
\end{proposition}
\begin{proof}
    Let $\alpha_1,\alpha_2,\beta_1,\beta_2> 0$ and $t\in [0,1]$. H\"{o}lder's inequality gives
    \begin{align*}
        \int& \exp\Big(- (t\alpha_1+(1-t)\alpha_2)\sum_{k=0}^{n-1}V(\Delta_k\gamma +v)-(t\beta_1 + (1-t)\beta_2)\sum_{k=0}^{n-1}F_k(\gamma_k) \Big)\dx{\gamma}\\
        & = \int \exp\Big( - t\alpha_1\sum_{k=0}^{n-1}V(\Delta_k\gamma +v)-t\beta_1 \sum_{k=0}^{n-1}F_k(\gamma_k)\Big) \\
        &\qquad\qquad\qquad\times \exp\Big( - (1-t)\alpha_2 \sum_{k=0}^{n-1}V(\Delta_k\gamma +v) -(1-t)\beta_2\sum_{k=0}^{n-1}F_k(\gamma_k) \Big)\dx{\gamma}\\
        & \le \Big(\int \exp\Big(-\alpha_1\sum_{k=0}^{n-1}V(\Delta_k\gamma +v)-\beta_1\sum_{k=0}^{n-1}F_k(\gamma_k)\Big)\dx{\gamma}\Big)\Big)^t \\
        &\qquad\qquad\qquad\times \Big(\int \exp\Big(-\alpha_2\sum_{k=0}^{n-1}V(\Delta_k\gamma +v)-\beta_2\sum_{k=0}^{n-1}F_k(\gamma_k)\Big)\dx{\gamma}\Big)\Big)^{1-t}.
    \end{align*}
    Taking logs of both sides and dividing by $n$ establishes concavity of
    $ -\frac{1}{n}\log \Zshear^n(v,\alpha,\beta)$ in~$\alpha$ and~$\beta$. Taking $n\to\infty$ and applying
    \Cref{shapeTheoremPolymer} and the inequalities  analogous to those in \eqref{convexityInLimit}, we obtain that $\Lambda_1(v,\alpha,\beta)$ is concave in $(\alpha,\beta)$.
    
    Now we establish \eqref{empiricalKineticEnergyPolymer}. For $\beta,\alpha,\alpha'> 0$,
    \begin{align*}
         &\log \Zshear^n(v,\alpha',\beta) &
        \\
        =& \log \int \exp\bigg( - \alpha\sum_{k=0}^{n-1}V(\Delta_k \gamma + v)-\beta \sum_{k=0}^{n-1}F_k(\gamma_k) - (\alpha'-\alpha)\sum_{k=0}^{n-1}V(\Delta_k \gamma + v)\bigg)\dx\gamma\\
        =& \log \Zshear^n(v,\alpha,\beta) + \log \mushear_{v,\alpha,\beta}^n\Big[\exp\Big((\alpha'-\alpha)\sum_{k=0}^{n-1}V(\Delta_k \gamma + v)\Big)\Big]\\
        \ge& \log \Zshear^n(v,\alpha,\beta) + (\alpha'-\alpha)\mushear_{v,\alpha,\beta}^n\Big[\sum_{k=0}^{n-1}V(\Delta_k \gamma + v)\Big].
    \end{align*}
    We are going to apply \Cref{deterministicDifferentiability}, and specifically \eqref{concaveImpliesBounds}, to conclude \eqref{empiricalKineticEnergy}.
    Let~$\mathcal{D}$ be any countable dense subset of $\R.$ Fix some $\alpha,\beta >0.$ \Cref{shapeTheoremPolymer} implies that with $\Prb$-probability one, for all $x\in \mathcal{D}\cup \{\alpha\},$
    \[\lim_{n\to \infty} -\frac{1}{n}\log \Zshear^n(v,x,\beta) = \Lambda_1(v,x,\beta).\]
    Thus, with $\Prb$-probability one, the conditions of \Cref{deterministicDifferentiability} are satisfied with 
    \begin{align*}
        x_0 &= \alpha,\\
        f_n(x) &= -\frac{1}{n}\log \Zshear^n(v,x,\beta),\\
        f(x) &= \Lambda_1(v,x,\beta),\\
        g_n &= \mushear_{v,\alpha,\beta}^n(\overline{V_n}(v,\gamma)),\\
        h(x) &= 0,
    \end{align*}
    so the estimate \eqref{empiricalKineticEnergyPolymer} follows from \eqref{concaveImpliesBounds} of \Cref{deterministicDifferentiability}.
    The estimate \eqref{empiricalPotentialEnergyPolymer} follows similarly.
\end{proof}

Let us introduce
 \[
 N_{v,\infty} =  \limsup_{n\to \infty}\frac{1}{n}\mushear_v^n\bigg(\sum_{k=0}^n \sup_{|r|\le 1}V''(\Delta_k\gamma + v + r)\bigg). 
 \]
\begin{lemma}\label{boundedSecondDerivSumsPolymer}
    For every $v\in \R$,
    \begin{equation}\notag
        N_{v,\infty} < \infty 
    \end{equation}
    with $\Prb$-probability one.
\end{lemma}
\begin{proof}
    Inequality \eqref{upperBoundL} in the proof of~\Cref{boundedSecondDerivSums} implies that for some $A,C>0$ and some $b\le 0,$
    \begin{equation*}
        \frac{1}{n}\mushear_v^n\bigg(\sum_{k=0}^n \sup_{|r|\le 1}V''(\Delta_k\gamma + v + r)\bigg) \le A + C\mushear_v^n(\overline{V_n}(v,\gamma)) - Cb,
    \end{equation*}
    so our claim follows from  \Cref{concavityOfShapeFcnPolymer} applied to $\tilde \mu^n_{v,1,1}$. 
\end{proof}

With these auxiliary lemmas in hand, we can prove the main differentiability result for directed polymers.
\begin{proof}[Proof of \Cref{thm:differentiabilityPolymer}]
    Let $w,v\in \R$ satisfy $|w-v|\le 1.$ Inequality \eqref{taylorTheorem} along with Jensen's inequality implies 
    \begin{align*}
        &\log \Zshear^n(w)  = \log \int \exp\Big(-\sum_{k=0}^{n-1}\Big[F_k(\gamma_k)+ V(\Delta_k \gamma + w)\Big]\Big)\dx{\gamma}\\
        & \ge \log \int e^{-\sum_{k=0}^{n-1}\Big[F_k(\gamma_k)+ V(\Delta_k \gamma + v) + (w-v)V'(\Delta_k \gamma +v) + \frac{1}{2}(w-v)^2 \sup_{|r|\le 1}V''(\Delta_k \gamma + v + r)\Big]}\dx{\gamma}\\
        & = \log \Zshear^n(v) + \log \frac{1}{\Zshear^n(v)}\int e^{-\sum_{k=0}^{n-1}\Big[F_k(\gamma_k)+ V(\Delta_k \gamma + v) + (w-v)V'(\Delta_k\gamma +v) + \frac{1}{2}(w-v)^2 \sup_{|r|\le 1}V''(\Delta_k \gamma + v + r)\Big]}\dx{\gamma}\\
        & = \log \Zshear^n(v) + \log \mushear_v^n \bigg(\exp\Big(-\sum_{k=0}^{n-1}\Big[(w-v)V'(\Delta_k\gamma +v) +\frac{1}{2}(w-v)^2 \sup_{|r|\le 1}V''(\Delta_k \gamma + v + r) \Big]\Big)\bigg)\\
        & \ge \log \Zshear^n(v) - (w-v) \mushear_v^n\bigg( \sum_{k=0}^{n-1}V'(\Delta_k\gamma +v)\bigg) - \frac{1}{2}(w-v)^2\mushear_v^n\bigg(\sum_{k=0}^{n-1} \sup_{|r|\le 1}V''(\Delta_k\gamma + v + r)\bigg).
    \end{align*}
   \Cref{boundedSecondDerivSumsPolymer} implies that for sufficiently large $n$,
    \begin{equation}\notag
        -\frac{1}{n}\log\Zshear^n(w) \le -\frac{1}{n}\log \Zshear^n(v) + (w-v)\mushear_v^n(\overline{V'_n}(v,\gamma)) +  \frac{1}{2}(w-v)^2 (N_{v,\infty}+1).
    \end{equation}
   As in the zero temperature case, the above implies that $-\frac{1}{n}\log \Zshear^n(\cdot)$ satisfies the approximate linear domination condition stated in \eqref{finiteDerivativeInequality} at $v\in \R$, with 
    \begin{align*}
        x_0 &= v,\\
        f_n(x) &= -\frac{1}{n}\log \Zshear^n(x),\\
        g_n & = \mushear_v^n(\overline{V'_n}(v,\gamma)),\\
        h(x) &= \frac{1}{2}(v-x)^2 (N_{v,\infty}+1).
    \end{align*}
Now our theorem follows from  Lemma~\ref{deterministicDifferentiability}.  To apply  Lemma~\ref{deterministicDifferentiability}, it remains to take $f = \Lambda_1$ (which is convex due to 
 \Cref{shapeTheoremPolymer}), take any countable dense set $\mathcal{D}\subset \R$, 
 and notice (see Remark~\ref{shapeTheoremShear}) that 
 on an event of probability one, \eqref{eq:shapeTheoremBnPolymer} holds
    for all $x\in \mathcal{D}\cup \{v\}$. This completes the proof.
   \end{proof}

\section{The Shape Theorems}\label{shapeTheoremSection}

\subsection{Proof of \Cref{shapeTheorem}}

The proof of \Cref{shapeTheorem} is almost identical to Lemma 4.7 in \cite{kickb:bakhtin2016}. See also Theorem 2.1 in \cite{AuffingerDamronHanson_50Years:MR3729447} and Theorem 2.18 
in~\cite{10.1007/BFb0074919} for examples of similar proofs in the fully discrete setting.
To prove the existence of the limit in \eqref{eq:shapeFunction}, let
us fix $v\in\R$.
For $m,n\in\Z$ with $m<n$, let
\begin{equation}
\label{eq:Amnv}
A^{m,n}_*(v) = \inf \{A^{m,n}(\gamma)\,:\, \gamma\in \Gamma^{m,n}(v)\},
\end{equation}
where $A^{m,n}$ is defined in \eqref{eq:path_energy} and 
\[\Gamma^{m,n}(v) = \{\gamma\in\R^{n-m+1}\,:\,\gamma_m = vm,\,\gamma_n = vn\}.\]
Our conditions on $F$ and $V$ imply that for every
$m,n,v$
there is a path $\gamma^{m,n}(v)$ realizing 
the infimum in
\eqref{eq:Amnv} (see Lemma~\ref{lem:existenceOfMinimizer}), and 
\begin{equation}
\label{eq:lowerb_on_Amnv}
\frac{1}{n-m}A^{m,n}_*(v)\ge M_F+M_V>-\infty,
\end{equation}
where $M_F$ and $M_V$ are lower bounds on $F$ and $V$ introduced in Section~\ref{sec:Setting_MainResults}.

Concatenating minimizers $\gamma^{0,m}(v)$ and $\gamma^{m,m+n}(v)$, we obtain the subadditivity property for $A_*^{0,m}(v)$:
\begin{equation}\label{eq:subadditivity}
    A_*^{0,n+m}(v) \le A_*^{0,m}(v) + A_*^{m,n+m}(v),\quad  m,n\in \N.
\end{equation}
 Consider the map $T_v:\Omega\to \Omega$ defined by 
\[(T_v F)_k = F_{k+1}(x - v).\]
We have the following skew-invariance property of the action under the semigroup generated by $T_v$:
\[A_*^{m,n+m}(v)(F) = A_*^{0,n}(v)(T_v^m F).\]
Our conditions on the environment $\Omega$ (stationarity in space and the i.i.d.\ property in time) imply that the map $T_v$ preserves $\Prb$ and that it is ergodic. 

We have
\begin{equation}\notag
\E\, |A_*^{0,n}(v)| < \infty,
\end{equation}
which follows from~\eqref{eq:lowerb_on_Amnv}, \eqref{eq:subadditivity} and equality
\[
A^{0,1}_* = V(v)+ F_0(0).
\]
Thus,
Kingman's subadditive ergodic theorem implies the existence of 
a deterministic $\Prb$-a.s.\ limit 
$\Lambda_0(v)$ in \Cref{shapeTheorem} and the estimate  $\Lambda_0(v)\ge M_F+M_V>-\infty$.

\medskip 
Let us prove that $\Lambda_0$ is convex.
We need to check that for all $v_1,v_2\in\R$ and $t\in(0,1)$,
\begin{equation}
\label{eq:def_convex}
\Lambda_0(v)\le t \Lambda_0 (v_1)+(1-t)\Lambda_0(v_2),  
\end{equation}
where 
\[
v=tv_1+(1-t)v_2.
\]
First let us prove this for $v_1,v_2\in\Q$ and $t\in(0,1)\cap\Q$.

Let us take a sequence of integers $n_k\uparrow\infty$ such that
$tn_k\in\N$ for all $k$.
Concatenating optimal paths, we obtain
\begin{equation}\notag
A^{n_k}_*(n_kv) \le A_*^{tn_k }(t n_kv_1) 
+ A^{tn_k,n_k}_{tn_kv_1, n_k v}.
\end{equation} 
Dividing by $n_k$ and taking limits in probability as $k\to\infty$, we obtain \eqref{eq:def_convex} for rational values of parameters.
To check it for general values of parameters it suffices to prove that $\Lambda_0$ is continuous. 
Let us fix $v\in\R$ and check continuity of $\Lambda_0$ at $v$.

We will need an arbitrary function $b:(v-1,v+1)\to \Q$
satisfying $\lim_{u\to v}b(u)=1$ and $b(u)-1\ge \sqrt{|u-v|}$.

For any $u\in(v-1,v+1)$, let us consider an integer sequence $n_k\uparrow\infty$ satisfying $bn_k\in\N$ 
and $n_{k+1}>bn_k$
for all $k$.
For brevity, we will write $n=n_k$ and $b=b(u)$ from now on. Concatenating minimizers, we obtain 
\[
A^{0,bn}_{0,ubn}\le A_{0,vn}^n+A^{n,bn}_{vn,ubn}
\]
or
\begin{equation}
\label{eq:lower_continuity_approximation_1}
A_{0,vn}^n\ge A^{0,bn}_{0,ubn} - A^{n,bn}_{vn,ubn}.
\end{equation}
Let us estimate the second term on the right-hand side.

To that end, we consider a straight path $\gamma^{(n)}$ defined by
\[
\gamma^{(n)}_i=vn+\frac{ub-v}{b-1}(i-n),\quad i=n,\ldots,bn,
\]
It satisfies 
\begin{align*}
\Delta_i(\gamma^{(n)})&=\frac{ub-v}{b-1},\quad  i=n,\ldots,bn-1,
\\
\gamma^{(n)}_{n}&=vn,\\
\gamma^{(n)}_{bn}&
=ubn.
\end{align*}
We have
\begin{equation}
\label{eq:lower_continuity_approximation_2}
\frac{1}{n}A^{n,bn}_{vn,ubn}\le\frac{1}{n} A^{n,bn}(\gamma^{(n)}) =(b-1)V\Big(\frac{ub-v}{b-1}\Big)+\frac{1}{n}\Sigma(u,n),
\end{equation}
where
\[
\Sigma(u,n)=\sum_{i=n}^{bn-1}F_i(\gamma^{(n)}_i).
\]
All terms in the definition of $\Sigma(u,n_k)$ are jointly i.i.d.,
over all $k$ due to our choice of $(n_k)_{k\in \N}$. The strong law of large numbers implies then that
$\Sigma(u,n)$ satisfies 
\[
\lim_{k\to\infty}\frac{1}{n}\Sigma(u,n)=(b-1)\E F_0(0)
\]
$\Prb$-almost surely.
The first term on the right-hand side of \eqref{eq:lower_continuity_approximation_2} can be estimated by
\[
(b-1)V\Big(\frac{ub-v}{b-1}\Big)= 
(b-1)V\Big(u+\frac{u-v}{b-1}\Big)\le c (b-1)
\]
for some $c>0$
due to our assumptions on $v,u,b$ and $V$. Plugging these estimates into
\eqref{eq:lower_continuity_approximation_2}, we obtain that 
with probability 1
\begin{equation*}
\limsup_{k\to\infty} \frac{1}{n}A^{n,bn}_{vn,ubn}\le C(b-1),
\end{equation*}
where $C=c+\E F_0(0).$
Thus, dividing~\eqref{eq:lower_continuity_approximation_1} by $n_k$,
and taking $\liminf_{k\to\infty}$, we obtain
\[
\Lambda_0(v)\ge b \Lambda_0(u)-C(b-1). 
\] 
Taking $u\to v$ (this implies $b\to 1$), we obtain
\[
\Lambda_0(v)\ge \limsup_{u\to v} \Lambda_0(u). 
\]
The proof of the matching upper bound 
\[
\Lambda_0(v)\le \liminf_{u\to v} \Lambda_0(u)
\]
is similar. It is based on the estimate
\[
A_{0,vn}^n\le
A^{0,bn}_{0,ubn}+A^{bn,n}_{ubn,vn}
\]
for a $\Q$-valued function $b$ satisfying
$b(u)-1\le-\sqrt{|u-v|}$ and $\lim_{u\to v}b(u)=1$. 
This completes the proof of convexity of $\Lambda_0$.
\epf

\subsection{Proof of \Cref{shapeTheoremPolymer}}
Our proof of the existence of $\Lambda_1$ is very similar to the proof of Lemma 6.2 and 6.3 in \cite{Bakhtin-Li:MR3911894}. See also Theorem~9.1 in \cite{Com17} or Theorem~2.2 in \cite{Rassoul-Agha--Seppalainen:MR3176363} for examples of proofs to the analogous proofs to \Cref{shapeTheoremPolymer} in the fully discrete setting.

Let 
\begin{equation}\notag
    Z_*^{m,n}(x,y) = \inf_{|x_1|,|y_1|<1/2} Z^{m,n}_{x+x_1,y+y_1},
\end{equation}
where $Z^{m,n}_{x,y}$ is defined in~\eqref{eq:p2p_pf}.
Let $Z_*^n(v)=Z_*^{0,n}(0,vn)$.

First we prove a bound on $\E|\log Z_*^{n}(v)|$. 
    Recalling the notation from \eqref{eq:ref-measure}, we obtain
    
    \begin{align} \label{ZStarLowerBound} \notag
        Z_*^n(v)  \ge& \inf_{|x|,|y|< 1/2} \int_{\R\times \prod_{k=1}^{n-1}[vk-1/2,vk+1/2]\times \R}  e^{-
A^{0,n}(\gamma)} \refpm_{x,vn+y}^{0,n}(d\gamma)
\\        & \ge  e^{-\sum_{k=0}^{n-1}F_k^*(vk) - n V^*},
    \end{align}
    where 
    \[V^* = \sup\Big\{V(z)\,:\,|z| \le |v| + 1\Big\},\]
    and $F_k^*$ is defined in \eqref{Fkmax}. Therefore, 
    \[-\log Z_*^{n}(v) \le \sum_{k=0}^{n-1} F_k^*(vk) + n V^*.\]    
    Also,
    \begin{align}\label{partitionFunctionLowerBound}
        \log Z^{0,n}_{x,y} &\le -n M_F + \log \int e^{-\sum_{k=0}^{n-1}V(\Delta_k \gamma)}\refpm_{x,y}^{0,n}(d\gamma) \notag \\
        & \le -n M_F + \log \|\stepdensity\|_{L^\infty(\R)} + (n-1)\log \|\stepdensity\|_{L^1(\R)}.
    \end{align}
    It follows that 
    \begin{align}\notag
        \E|\log Z_*^{n}(v)| &= \E[-\log Z^n_*(v)\1_{Z^n_*(v)<1}] 
        + \E[\log Z^n_*(v)\1_{Z^n_*(v)\ge 1}]\\ \notag 
        &\le n\E|F^*_0(0)| + n |V^*| - nM_F + \log \|\stepdensity\|_{L^\infty(\R)} + (n-1)\log \|\stepdensity\|_{L^1(\R)} \\
        & < \infty.\notag
    \end{align}

    Inequality \eqref{partitionFunctionLowerBound} implies that there is a constant $C>0$ such that for all $n\in \N,$
    \[-\frac{1}{n}\log Z^n_*(v) > -C.\]
    The same argument as in the proof of Lemma 6.1 in \cite{Bakhtin-Li:MR3911894} immediately implies that the sequence $(Z_*^{m,n}(vm,vn))_{m,n\in \N}$ is supermultiplicative: 
    \[Z_*^{0,n+m}(0,v(n+m)) \ge Z_*^{0,m}(0,vm)Z_*^{m,n+m}(vm,v(n+m)).\]
 The estimates obtained above allow us to apply  Kingman's subadditive ergodic theorem and conclude that the sequence $(-\frac{1}{n}\log Z_*^n(v))_{n\in \N}$ converges $\Prb$-almost surely to a limit that we will call $\Lambda_1(v)$.

Convexity of $\Lambda_1$ is established similarly to that of $\Lambda_0$. 
For rational $v_1,v_2,v\in\R$ and $t\in(0,1)$, satisfying
$v=tv_1+(1-t)v_2$, the convexity definition
\begin{equation}
\label{eq:def_convex_1}
\Lambda_1(v)\le t \Lambda_1 (v_1)+(1-t)\Lambda_1(v_2),  
\end{equation}
follows from applying $-\frac{1}{n}\log(\cdot)$ and taking $n$ satisfying 
$nt\in\N$ to $\infty$ in
\[Z_*^{0,n}(0,nv)) \ge Z_*^{0,tn}(0,tnv_1)Z_*^{tn,n}
(nv_1, nv).\]
 For general values of parameters, the estimate \eqref{eq:def_convex_1} follows now from continuity of~$\Lambda_1$ which is also established similarly to that of $\Lambda_0$. 

Theorem~\ref{shapeTheoremPolymer} now immediately follows from 
the following lemma.\epf

\begin{lemma}\label{exponentialApproximationLemma}
    For every $v\in \R,$
    \begin{equation}\label{subexponentialApproximation}
        \lim_{n\to \infty}\frac{1}{n}\Big|\log Z^n(v) - \log Z_*^n(v)\Big| = 0
    \end{equation}
    $\Prb$-almost surely.
\end{lemma}

Our proof is similar to the proof of Lemma 6.3 in \cite{Bakhtin-Li:MR3911894}.
\begin{proof}[Proof of \Cref{exponentialApproximationLemma}]
    By definition, $Z^n(v) \ge Z_*^{n}(0,vn).$ Thus, it suffices to prove 
    \[\limsup_{n\to \infty}\sup_{|x|,|y|<1/2} \frac{1}{n}\log \frac{Z^n(v)}{Z_{x,vn+y}^n}  = 0.\]
    Inequality \eqref{ZStarLowerBound} and the integrability of $F_k^*(vk)$ implies that there is a deterministic $q > 0$ such that 
    \begin{equation}\label{exponentialLowerBoundZStar}
        \liminf_{n\to \infty}\frac{Z_*^{n}(0,vn)}{q^n} > 0
    \end{equation}
    $\Prb$-almost surely. 

    Let $r^+ = r_n^+$ and $r^- = r_n^-$ be numbers depending on $n$ that we will specify later and let $I = I_n = [-r^- , r^+]$. We have 
    \begin{align*}
        Z^n(v) &=  \int_{\R\times \R} Z^{0,1}_{0,x'}Z^{1,n-1}_{x',y'}Z^{n-1,n}_{y',vn}dx'dy'\\
        & \le W_{1,-}^n + W_{1,+}^n + W_{2,-}^n + W_{2,+}^n + W_3^n,
    \end{align*}
    where
    \begin{align*}
        W_{1,\pm}^n &= \int_{\pm x'>r^\pm} Z^{0,1}_{0,x'}Z^{1,n-1}_{x',y'}Z^{n-1,n}_{y',vn}dx'dy',\\
        W_{2,\pm}^n &= \int_{\pm(vn-y')>r^\pm} Z^{0,1}_{0,x'}Z^{1,n-1}_{x',y'}Z^{n-1,n}_{y',vn}dx'dy',\\
        W_3^n &= \int_{x',(vn-y')\in I}Z^{0,1}_{0,x'}Z^{1,n-1}_{x',y'}Z^{n-1,n}_{y',vn}dx'dy'.
    \end{align*}    
    We now give conditions under which $W_{1,\pm}^n$ and $W_{2,\pm}^n$ decay super-exponentially in~$n$. We will use the following lemma, whose proof we postpone.
    \begin{lemma}\label{tailDecayV}
        If $V$ satisfies \eqref{linearGrowthV}, then there are constants $C,K>0$ such that for all $x>K$ 
        \begin{equation}\label{eq:tailDecay}
            \int_{y>x}e^{-V(y)}dy \le C e^{-V(x)}\quad \textrm{ and }\quad  \int_{y<-x}e^{-V(y)}dy \le C e^{-V(-x)}.
        \end{equation}
    \end{lemma}
    Let $p^{\ast k}$ be the $k$-fold convolution of $p = e^{-V}$.  \Cref{tailDecayV} implies that if $r^+>K$, then 
    \begin{align*}
        \E W_{1,+}^n &= (\E e^{-F_0(0)})^n \int_{x' > r^+}p(x') p^{\ast (n-1)}(vn-x')dx'\\
        & \le (\E e^{-F_0(0)})^n \|p^{\ast (n-1)}\|_{L^\infty(\R)}\int_{x'>r^+}p(x')dx'\\
        & \le C (\E e^{-F_0(0)})^n \|p\|_{L^\infty(\R)}\|p\|_{L^1(\R)}^{n-2} p(r^+).
    \end{align*}
   We recall that $q$ is chosen to ensure \eqref{exponentialLowerBoundZStar}. There are constants  $C',D > 0$ such that
    \begin{equation}\label{W1largeBound}
        \Prb\big\{|W_{1,+}^n| > (q/2)^n\big\} \le (q/2)^{-n}\E W_{1,+}^n \le C' e^{Dn} p(r_n^+).
    \end{equation}
    If $r_n^\pm$ are chosen to guarantee
    \begin{equation}\label{conditionOnR1}
        p(r_n^\pm) = e^{-V(r_n^\pm)} = O(e^{-R n}),\quad n\to\infty,\quad \forall R>0,
    \end{equation}
 then \eqref{W1largeBound} is summable and so the Borel--Cantelli Lemma implies that 
    \[\limsup_{n\to \infty}\frac{W_{1,+}^n}{Z_*^n(0,vn)} = \limsup_{n\to \infty}\frac{W_{1,+}^n}{q^n}\Big(\frac{Z_*^n(0,vn)}{q^n}\Big)^{-1} = 0\]
    $\Prb$-almost surely. Similar analysis shows that \eqref{conditionOnR1} implies 
    \[\limsup_{n\to \infty}\frac{W^n_{1,-}}{Z_*^n(0,vn)} = \limsup_{n\to \infty}\frac{W_{2,\pm}^n}{Z_*^n(0,vn)} = 0\]
    $\Prb$-almost surely.

    Now we consider $W_3^n$. Simple manipulation shows that 
    \begin{align}\label{W3ratioupperBound}
        \sup_{|x|,|y|< 1/2}\frac{W_3^n}{Z^{n}_{x,vn+y}} &\le \sup_{\stackrel{|x|,|y|< 1/2}{x',(vn-y') \in I}}\frac{p(x')p(vn-y')e^{-F_0(0)}}{p(x'-x)p(vn+y-y') e^{-F_0(x)}}\notag\\
        & = e^{-F_0(0) + \sup_{|x|< 1/2} F_0(x)}\exp\Big(2 \sup_{|x|< 1/2,x'\in I} [V(x'-x) - V(x')]\Big)\notag\\
        & \le e^{-F_0(0) + \sup_{|x|< 1/2}F_0(x)}\exp\Big(2\sup_{x\in I+(-1/2,1/2)} |V'(x)|\Big).
    \end{align}
    By \eqref{linearGrowthV}, \eqref{growthConstraintV}, and \eqref{derivativeRatio}, there are $c,C_1,K_1 > 0,$ $\theta\in(0,1)$,  such that if $|x|>K$ then $|V'(x) |\ge c$, $|V'(x)| \le C_1 |V(x)|^{\theta}$, and $|V(x\pm 1/2)| \le C_1 |V(x)|$. Denoting 
  $C_2= \sup_{|x|\le K}|V'(x)|$ and using monotonicity of $V$ on each component of  $\{x:|x|>K\}$, we obtain  
    \begin{align*}
        \sup_{x\in I+(-1/2,1/2)} |V'(x)| &\le \sup_{|x|\le K}|V'(x)| + \sup_{\stackrel{x\in I+(-1/2,1/2)}{|x|>K}}|V(x)|^{\theta}\notag\\
        & = C_2 + C_1|V(r^+ + 1/2)|^{\theta} + C_1 |V(- r^- - 1/2)|^{\theta}\notag\\
        & = C_2 + C_1^2 |V(r^+)|^{\theta} + C_1^2 |V(-r^-)|^{\theta}.
    \end{align*}
    Using this estimate in \eqref{W3ratioupperBound} we obtain that there is a random variable $C_3$  (that does not depend on $n$) and a constant $C_4$  such that
    \begin{equation}\label{W3ratioupperBound2}
        \frac{W_3^n}{Z_*^n(0,vn)} \le C_3 e^{C_4 \max(|V(r^+)|,|V(-r^-)|)^{\theta}}.
    \end{equation}
    If we choose $r^+ = r_n^+$ and $r^- = r_n^-$ such that 
    \begin{equation}\label{conditionOnR2}
        \max(|V(r_n^+)|,|V(-r_n^-)|) = o(n^{1/{\theta}})
    \end{equation}
    then \eqref{W3ratioupperBound2} will imply 
    \[\limsup_{n\to \infty}\frac{1}{n}\log \frac{W_3^n}{Z_*^n(0,vn)} = 0.\]

    This analysis shows that if we can find sequences $(r_n^+)_{n\in \N}$ and $(r_n^-)_{n\in\N}$ satisfying~\eqref{conditionOnR1} and \eqref{conditionOnR2}, then \eqref{subexponentialApproximation} will follow. Due to \eqref{kineticEnergyGrowthAtInfinity},  we can choose sequences $r_n^+\to \infty$ and $r_n^-\to \infty$ such that for sufficiently large $n\in\N,$
    \[V(r_n^+) = n^{\frac{1+1/\theta}{2}},\quad V(-r_n^-) = n^{\frac{1+1/\theta}{2}},\]
    so that \eqref{conditionOnR1} and \eqref{conditionOnR2} hold. This completes the proof of the lemma.
\end{proof}

\begin{proof}[Proof of \Cref{tailDecayV}]
    Let $C,K>0$ be such that $|V'(x)| \ge C$ for $|x|>K$. If $x>K$, then
    \begin{align*}
        \int_{y > x} e^{V(x)-V(y)}dy & = \int_{y > x} e^{- \int_x^yV'(z)dz}dy
         \le \int_{y > x} e^{-C(y-x)}dy \le \frac{1}{C}.
    \end{align*}
    Multiplying both sides by $e^{-V(x)}$ proves the right tail claim of \eqref{eq:tailDecay}, and the left tail claim follows similarly.
\end{proof}

\section{Proof of an Auxiliary Lemma}\label{auxiliaryProofs}
\begin{proof}[Proof of \Cref{lem:existenceOfMinimizer}]
    To prove that there is a minimizer to \eqref{eq:AF_def},
     it suffices, due to continuity of $A^n$, to check that
    \begin{equation*}
        \lim_{\substack{\max_k |\gamma_k|\to \infty\\ \gamma\in \Gamma^n(v)}}|A^n(\gamma)| = \infty.
    \end{equation*}
    This relation is a consequence of $\max_{k}|\Delta_k \gamma |\to\infty $, condition  \eqref{kineticEnergyGrowthAtInfinity}, and 
    the estimate
    \begin{equation*}
        \sum_{k=0}^{n-1}V(\Delta_k \gamma) + \sum_{k=0}^{n-1}F_k(\gamma_k) \ge \max_{k=0,\dots,n}V(\Delta_k \gamma) + (n-1)M_V + n M_F.
    \end{equation*}

    Now we prove existence a measurable selection of minimizer to \eqref{eq:AF_def}. Consider the set-valued function 
    \begin{align*}
        \psi:\Omega&\to \mathcal{P}(\Gamma^n(v))\\ 
        F&\mapsto \{\gamma\in \Gamma^n(v)\,:\, A^n(\gamma) = A_*^n(v)\},
    \end{align*}
    where $\mathcal{P}(\Gamma^n(v))$ is the power set of $\Gamma^n(v).$ 
    We identify $\Gamma^n(v)$ with $\R^{n-1}$ (the endpoints of paths are fixed) and equip it with Euclidean norm $\|\cdot\|_2$.

    We are going to apply the Kuratowski--Ryll-Nardzewski Selection Theorem (see Theorem 18.13 in \cite{Aliprantis:MR2378491}), which will allow 
    us to conclude that there is a measurable map $\gamma_A^n(v):\Omega\to \Gamma^n(v)$ satisfying
    \begin{equation}\label{gammaAnProperty}
        A^n(\gamma_A^n(v)) = A_*^n(v).
    \end{equation}
The conditions of that theorem requiring that for every $F\in \Omega,$ $\psi(F)$ is non-empty and closed hold true since 
$A^n(\gamma)$ is continuous in $\gamma$ and 
minimizers exist. To ensure the remaining condition of weak measurability of  $\psi$, we must  check that
 for every open set $U\subset \Gamma^n(v)$ the set 
    \[U_{\psi^{-1}} := \{F\in \Omega\,:\, \psi(F)\cap U \neq \emptyset\}\]
    is measurable in $\Omega.$ 
    
    Let $\mathcal{D}$ be any countable dense subset of $\Gamma^n(v).$ For a real number $r > 0$ let $U^r$ be the set of points $\gamma\in U$ such that if $\|\gamma-\gamma'\|_2 < r$, then $\gamma'\in U.$ We have 
    \begin{align*}
        U_{\psi^{-1}}&= \{F\in \Omega\,:\, \exists \gamma \in U,\,\, A^n(\gamma) = A_*^n(v)\}\\
        & = \{F\in \Omega\,:\, \exists r > 0\,\textrm{ s.t. }\exists (\gamma(k))_{k\in \N} \subset U^r\cap\mathcal{D},\,\, \lim_{k\to \infty}A^n(\gamma(k)) = A_*^n(v)\}\\
        & = \bigcup_{m\in \N}\bigcap_{k\in \N}\bigcup_{\gamma\in U^{\frac{1}{m}}\cap \mathcal{D}}\bigcap_{\gamma'\in \Gamma^n(v)\cap \mathcal{D}}
        C(k,\gamma,\gamma'),
    \end{align*}
    where
    \[
    C(k,\gamma,\gamma')
        =\Big\{F\in \Omega\,:\, A^n(\gamma) \le A^n(\gamma') + \frac{1}{k}\Big\}.
    \]
 For a fixed $\gamma\in\Gamma^n(v)$, $A^n(\gamma)$ is continuous as a function from $\Omega$ to $\R.$ Thus, for every $k,\gamma$, and $\gamma',$ 
 $C(k,\gamma,\gamma')$ is measurable, and so $U_{\psi^{-1}}$ is also measurable in $\Omega.$
  
Now 
the Kuratowski--Ryll-Nardzewski Selection Theorem implies the existence of a measurable map $\gamma_A^n(v):\Omega\to \Gamma^n(v)$ satisfying
~\eqref{gammaAnProperty}.

    Recall the map $\Xi_v^*$ defined in \eqref{FShearInvariant}. We have by definition 
    \begin{equation}\label{BnTransformation}
        B_*^n(v)(\Xi_v^*F) = A_*^n(v)( F).
    \end{equation}
    Define
    \[\gamma_B^n(v)(F) := \Xi_{-v}\gamma_A^n(v)(\Xi_{-v}^* F).\]
    Equalities \eqref{BnTransformation} and \eqref{gammaAnProperty} imply that 
    \[B^n(v,\gamma_B^n(v)) = B_*^n(v).\]
    Equality \eqref{shearedEnvironmentEquality} follows by choice of $\gamma_B^n(v)$ and the fact that $\Xi_v^*$ is measure preserving.
\end{proof}

\bibliographystyle{alpha} 
\bibliography{Burgers,polymer}

\end{document}